\documentclass[11pt,reqno]{amsart}
\usepackage{fullpage}
\usepackage{mathrsfs,amssymb,graphicx,verbatim,amsmath,amsfonts}
\usepackage{paralist}
\usepackage{enumitem}
\usepackage[breaklinks,pdfstartview=FitH]{hyperref}
\usepackage{upgreek}
\usepackage[mathscr]{euscript}

\renewcommand{\le}{\leqslant}
\renewcommand{\ge}{\geqslant}

\renewcommand{\setminus}{\smallsetminus}
\renewcommand{\gamma}{\upgamma}

\newcommand{\cupdot}{\mathbin{\mathaccent\cdot\cup}}

\renewcommand{\j}{\mathsf{j}}
\newcommand{\q}{\mathsf{q}}

\newcommand{\w}{\mathsf{w}}
\renewcommand{\emptyset}{\varnothing}
\newcommand{\vol}{\mathrm{\bf vol}}
\newcommand{\spn}{\mathrm{\bf span}}
\newcommand{\conv}{\mathrm{\bf conv}}
\newcommand{\rank}{\mathrm{\bf rank}}
\newcommand{\srank}{\mathrm{\bf srank}}
\newcommand{\entrank}{\mathrm{\bf Entrank}}
\renewcommand{\det}{\mathrm{\bf det}}

\renewcommand{\P}{\mathsf{P}}
\newcommand{\sign}{\mathrm{\bf sign}}

\newcommand{\n}{\{1,\ldots,n\}}
\newcommand{\m}{\{1,\ldots,m\}}

\newcommand{\s}{\sigma}

\renewcommand{\d}{\delta}
\newcommand{\D}{{\mathsf{D}}}

\newcommand{\e}{\varepsilon}
\newcommand{\R}{\mathbb R}
\newcommand{\1}{\mathbf 1}

\newtheorem{theorem}{Theorem}
\newtheorem{lemma}[theorem]{Lemma}

\theoremstyle{remark}
\newtheorem{remark}[theorem]{Remark}

\newtheorem{example}[theorem]{Example}

\makeatletter
\def\moverlay{\mathpalette\mov@rlay}
\def\mov@rlay#1#2{\leavevmode\vtop{%
   \baselineskip\z@skip \lineskiplimit-\maxdimen
   \ialign{\hfil$\m@th#1##$\hfil\cr#2\crcr}}}
\newcommand{\charfusion}[3][\mathord]{
    #1{\ifx#1\mathop\vphantom{#2}\fi
        \mathpalette\mov@rlay{#2\cr#3}
      }
    \ifx#1\mathop\expandafter\displaylimits\fi}
\makeatother

\newcommand{\bigcupdot}{\charfusion[\mathop]{\bigcup}{\cdot}}

\renewcommand{\S}{\mathsf{S}}
\renewcommand{\ss}{\mathsf{s}}
\renewcommand{\subset}{\subseteq}
\renewcommand{\supset}{\supseteq}

\DeclareMathOperator{\trace}{\bf Tr}
\newcommand{\E}{\mathbb{ E}}

\newcommand{\N}{\mathbb N}

\newcommand{\eqdef}{\stackrel{\mathrm{def}}{=}}
\newcommand{\proj}{\mathsf{Proj}}

\begin{document}


\title[Restricted invertibility revisited]{Restricted invertibility revisited}


\dedicatory{Dedicated to Jirka Matou\v{s}ek}

\author{Assaf Naor}
\address{Mathematics Department\\ Princeton University\\ Fine Hall, Washington Road, Princeton, NJ 08544-1000, USA}
\email{naor@math.princeton.edu}
\thanks{A.~N. was supported by BSF grant
2010021, the Packard Foundation and the Simons Foundation.}

\author{Pierre Youssef}
\address {Laboratoire de Probabilit\'es et de Mod\`eles Al\'eatoires\\
Universit\'e Paris-Diderot\\
5 rue Thomas Mann 75205, Paris CEDEX 13, France}
\email{youssef@math.univ-paris-diderot.fr}
\date{}

\begin{abstract} Suppose that $m,n\in \N$ and that $A:\R^m\to \R^n$ is a linear operator. It is shown here that if $k,r\in \N$ satisfy $k<r\le \rank(A)$ then there exists a subset $\s\subset \m$ with $|\s|=k$ such that the restriction of $A$ to $\R^\s\subset \R^m$ is invertible, and moreover the operator norm of the inverse $A^{-1}:A(\R^\s)\to \R^m$ is at most a constant multiple of the quantity $\sqrt{mr/((r-k)\sum_{i=r}^m \ss_i(A)^2)}$,
where $\ss_1(A)\ge\ldots\ge \ss_m(A)$ are the singular values of $A$. This improves over a series of works, starting from the seminal Bourgain--Tzafriri Restricted Invertibility Principle, through the works of Vershynin, Spielman--Srivastava and Marcus--Spielman--Srivastava. In particular, this directly implies an improved restricted invertibility principle in terms of Schatten--von Neumann norms.
\end{abstract}

\maketitle

\section{Introduction}

Given $m,n\in \N$, the rank of a linear operator $A:\R^m\to \R^n$ equals the largest possible dimension of a linear subspace $V\subset \R^m$ on which $A$ is injective, i.e., the inverse $A^{-1}:A(V)\to V$ exists. The {\em restricted invertibility problem} asks for conditions on $A$ that ensure a strengthening of this basic fact from linear algebra in two ways, corresponding to additional {\em structural information} on the subspace $V\subset \R^m$ on which $A$ is injective, as well as {\em quantitative information} on the behavior of the inverse $A^{-1}:A(V)\to V$. Firstly, the goal is to find a large dimensional {\em coordinate subspace} on which $A$ is invertible, i.e., we wish to find a large subset $\sigma\subset \m$ such that $A$ is injective on $\R^\sigma\subset \R^m$. Secondly, rather than being satisfied with mere invertibility we ask for $A$ to be  {\em quantitatively invertible}  on $\R^\sigma$ in the sense that the operator norm of the inverse $A^{-1}: A(\R^\sigma)\to \R^\s$ is not too large. Obviously, additional assumptions on $A$ are required for such conclusions to hold true.

The following theorem, which is known as the Bourgain--Tzafriri Restricted Invertibility Principle~\cite{BT87,BT89,BT91}, is a seminal result that addressed the above question and had major influence on subsequent research, with a variety of interesting applications to several areas. Throughout what follows, for $m\in \N$ the standard coordinate basis of $\R^m$ will be denoted by $e_1,\ldots,e_m\in \R^m$.

 \begin{theorem}[Bourgain--Tzafriri]\label{th-bourgain-tzafriri} There exist two universal constant $c,C\in (0,\infty)$ with the following property. Suppose that $m\in \N$ and that $A:\R^m\to \R^m$ is a linear operator such that the Euclidean norm of the vector $Ae_j\in \R^m$ equals $1$ for every $j\in \m$. Letting $\|A\|$ denote the operator norm of $A$, there exists a subset $\sigma\subset \m$ with $|\sigma|\ge cm/\|A\|^2$ such that $A$ is injective on $\R^\sigma$ and the operator norm of the inverse $A^{-1}:A(\R^\sigma)\to \R^\s$ is at most $C$.
 \end{theorem}

In what follows, for $p\in [1,\infty]$ and $m\in \N$ the $\ell_p$ norm of a vector $x\in \R^m$ will be denoted as usual by $\|x\|_{p}$. Thus $\|x\|_2$ is the Euclidean norm of $x$.  We shall also denote (as usual) by $\ell_p^m$ the normed space $\R^m$ equipped with the $\ell_p$ norm. The standard scalar product on $\R^m$ will be denoted $\langle\cdot,\cdot\rangle$. For $k,m,n\in \N$ and a $k$-dimensional subspace $V\subset \R^m$, the Schatten--von Neumann $p$ norm of a linear operator  $A:V\to \R^n$ will be denoted below by $\|A\|_{\S_p}$. Thus
$$
\|A\|_{\S_p}\eqdef \left(\trace(A^*A)^{\frac{p}{2}}\right)^{\frac{1}{p}}= \bigg(\sum_{j=1}^{k} \ss_j(A)^{p}\bigg)^{\frac{1}{p}},
$$
where  $\ss_1(A)\ge \ss_2(A)\ge \ldots\ge \ss_k(A)$ denote the singular values of $A$, i.e., they are the (decreasing rearrangement of the) eigenvalues of the positive semidefinite operator $\sqrt{A^*A}:V\to V^*$. Thus $\|A\|_{\S_\infty}=\ss_1(A)$ is the operator norm of $A$. Also, $\|A\|_{\S_2}$ is the Hilbert--Schmidt norm of $A$, i.e., for every orthonormal basis $u_1,\ldots,u_k$ of $V$ we have $\|A\|_{\S_2}^2=\sum_{i=1}^k\sum_{j=1}^n \langle Au_i,e_j\rangle^2=\sum_{i=1}^k\|Ae_i\|_2^2$.  Below it will sometimes be convenient to denote the smallest singular value of $A$ by $\ss_{\min}(A)=\ss_k(A)$. Thus $A$ is injective if and only if $\ss_{\min}(A)>0$, in which case $\|A^{-1}\|_{\S_\infty}=1/\ss_{\min}(A)$.

Given $m\in \N$ and $\sigma\subset \m$ it will be convenient to denote the formal identity from $\R^\sigma$ to $\R^m$ by $J_\s:\R^\sigma\to \R^m$, i.e., $J_\s((a_j)_{j\in \sigma})=\sum_{j\in \sigma} a_je_j$ for every $(a_j)_{j\in \sigma}\in \R^\s$. With this notation, given an operator $A:\R^m\to \R^n$ that is injective on $\R^\sigma$ we can consider the operator $(AJ_\sigma)^{-1}:A(\R^\sigma)\to \R^\s$. We shall sometimes drop the need to mention explicitly that $A$ is injective on $\R^\s$ by adhering to the convention that if $A$ is not injective on $\R^\sigma$ then $\|(AJ_\s)^{-1}\|_{\S_\infty}=\infty$.

Using the above notation,  Theorem~\ref{th-bourgain-tzafriri}  asserts that if $A:\R^m\to \R^m$ is a linear operator that satisfies $\|Ae_j\|_2=1$ for all $j\in \m$ then there exists  $\s\subset \m$ with $|\s|\gtrsim m/\|A\|_{\S_\infty}$ such that  $\|(AJ_\s)^{-1}\|_{\S_\infty}\lesssim 1$, or equivalently $\ss_{\min}(AJ_\s)\gtrsim 1$. Here, and in what follows, we use the following standard asymptotic notation. Given two quantities $K,L\in \R$ the notation $K\lesssim L$ (respectively $K\gtrsim L$) means that there exists a universal constant $c\in (0,\infty)$ such that $K\le cL$ (respectively $K\ge cL$). The notation $K\asymp L$ means that both $K\lesssim L$ and $K\gtrsim L$ hold true.

The following theorem is a useful strengthening of the Bourgain--Tzafriri Restricted Invertibility Principle that was discovered by Vershynin in~\cite{Ver01}.

\begin{theorem}[Vershynin]\label{th-vershynin} There exists a universal constant $c\in (0,\infty)$ with the following property. Fix $k,m,n\in \N$. Let $A:\R^m\to \R^n$ be a linear operator with $\|Ae_j\|_2=1$ for all $j\in \m$. Also, let $\Delta:\R^n\to \R^n$ be a positive definite diagonal operator, i.e., there exist $d_1,\ldots,d_n\in (0,\infty)$ such that $\Delta x=(d_1x_1,\ldots,d_n x_n)$ for every $x=(x_1,\ldots,x_n)\in \R^n$. Suppose that $k<\|A\Delta\|_{\S_2}^2/\|A\Delta\|_{\S_\infty}^2$ and write $k=(1-\e)\|A\Delta\|_{\S_2}^2/\|A\Delta\|_{\S_\infty}^2$ where $\e\in (0,1)$ (thus $\e=1-k\|A\Delta\|_{\S_\infty}^2/\|A\Delta\|_{\S_2}^2$).  Then there exists a subset $\s\subset \m$ with $|\s|=k$ such that $\|(AJ_\s)^{-1}\|_{\S_\infty}\le \e^{-c\log(1/\e)}$.
\end{theorem}

For  a linear operator $T:\R^m\to \R^n$, the quantity $\|T\|_{\S_2}^2/\|T\|_{\S_\infty}^2$ is often called the {\em stable rank} of $T$, though  this terminology sometimes also refers to the quantity $\|T\|_{\S_1}/\|T\|_{\S_\infty}$. In both cases, the use of the term `stable' in this context expresses the fact that the quantity in question is a robust replacement for the rank of $T$ in the sense that the rank of $T$ could be large due to the fact that $T$ has many positive but nevertheless very small singular values, while if the stable rank of $T$ is large then its singular values are large on average. Below we shall use the terminology `stable rank' exclusively for the quantity $\|T\|_{\S_2}^2/\|T\|_{\S_\infty}^2$, which we denote by $\srank(T)=\|T\|_{\S_2}^2/\|T\|_{\S_\infty}^2$.

Theorem~\ref{th-bourgain-tzafriri} coincides with the special case $\e=\frac12$ and $\Delta=I_n$ of Theorem~\ref{th-vershynin}, where $I_n$ is the identity operator on $\R^n$. However,  Theorem~\ref{th-vershynin}  improves over Theorem~\ref{th-bourgain-tzafriri} in three ways that are important for geometric applications. Firstly, Theorem~\ref{th-vershynin} treats rectangular matrices while Theorem~\ref{th-bourgain-tzafriri} treats only the case $m=n$. Secondly, even in the special case $\Delta=I_n$ of Theorem~\ref{th-vershynin} the size of the subset $\s\subset \m$ is allowed to be arbitrarily close to $\srank(A)$, while in Theorem~\ref{th-bourgain-tzafriri}  it can only be taken to be a  constant multiple of $\srank(A)$. Lastly, Theorem~\ref{th-vershynin} actually allows for  the size of the subset $\s\subset \m$ to be arbitrarily close to the supremum of $\srank(A\Delta)$ over all positive definite diagonal operators $\Delta:\R^m\to \R^m$, a  quantity that could be much larger than $\srank(A)$.

\begin{remark}
Theorem~\ref{th-vershynin} is often stated in the literature as a subset selection principle for John decompositions of the identity. Namely, suppose that $k,m,n\in \N$ and $x_1,\ldots,x_m\in \R^n\setminus \{0\}$ satisfy $\sum_{j=1}^m \langle x_j,y\rangle^2=\|y\|_2^2$ for all $y\in \R^n$. Equivalently, we have $\sum_{j=1}^m x_j\otimes x_j=I_n$, where  for $x,y\in \R^n$ the rank-one operator $x\otimes y:\R^n\to \R^n$ is defined as usual by setting $(x\otimes y)(z)=\langle x,z\rangle y$ for every $z\in \R^n$. Suppose that $T:\R^n\to \R^n$ is a linear operator satisfying $Tx_1,\ldots,Tx_m\neq 0$, and that $k=(1-\e)\srank(T)$ for some $\e\in (0,1)$. Then there exists $\sigma\subset \m$ with $|\s|=k$ such that
$$
\forall\, \{a_j\}_{j\in \s}\subset \R,\qquad \bigg\|\sum_{j\in \s} \frac{a_j}{\|Tx_j\|_2}Tx_j \bigg\|_2\ge  \e^{c\log(1/\e)}\bigg(\sum_{j\in \s} a_j^2\bigg)^{\frac12}.
$$
The above formulation is equivalent to Theorem~\ref{th-vershynin} as stated in terms of rectangular matrices by considering the operator $A:\R^m\to \R^n$ that is given by $Ae_j=Tx_j/\|Tx_j\|_2$ for every $j\in \m$.
\end{remark}

A recent breakthrough of Spielman--Srivastava~\cite{SS12}, that relies nontrivially on a remarkable method for sparsifying quadratic forms that was developed by Batson--Spielman--Srivastava~\cite{BSS12} (see also the survey~\cite{Nao12}), yielded the following improved restricted invertibility principle, via techniques that are entirely different from those used by Bourgain--Tzafriri and Vershynin.

\begin{theorem}[Spielman--Srivastava]\label{th-spielman-srivastava} Suppose that $k,m,n\in \N$ and let $A:\R^m\to \R^n$ be a linear operator such that $k<\srank(A)$. Write $k=(1-\e)\srank(A)$ where $\e\in (0,1)$. Then there exists a subset $\s\subset \m$ with $|\s|=k$ such that
$$
\|(AJ_\s)^{-1}\|_{\S_\infty}\le \frac{1}{1-\sqrt{1-\e}}\cdot \frac{\sqrt{m}}{\|A\|_{\S_2}}\le\frac{2\sqrt{m}}{\e\|A\|_{\S_2}}.
$$
\end{theorem}
In the setting of Theorem~\ref{th-spielman-srivastava},  since $\|A\|_{\S_2}=\sqrt{m}$ when the columns of $A$ have unit Euclidean norm,   Theorem~\ref{th-bourgain-tzafriri} is a special case of Theorem~\ref{th-spielman-srivastava}. As in the case $\Delta=I_n$ of Theorem~\ref{th-vershynin}, the statement of Theorem~\ref{th-spielman-srivastava} has the additional feature that the subset $\s\subset\m$ can have size arbitrarily close to $\srank(A)$. Moreover, in  Theorem~\ref{th-spielman-srivastava} the columns of $A$ need not have unit Euclidean norm, and the upper bound on $\|(AJ_\s)^{-1}\|_{\S_\infty}$ in terms of $\e$ is much better in  Theorem~\ref{th-spielman-srivastava} than the corresponding bound in the case $\Delta=I_n$ of Theorem~\ref{th-vershynin}; in fact this bound is asymptotically sharp~\cite{BHKW88} as $\e\to 0$. An additional feature of Theorem~\ref{th-spielman-srivastava} is that its proof in~\cite{SS12} yields a deterministic polynomial time algorithm for finding the subset $\s$, while previous to~\cite{SS12} only a randomized polynomial time algorithm was available~\cite{Tro09}. Theorem~\ref{th-vershynin} does have a feature that Theorem~\ref{th-spielman-srivastava}  does not, namely the size of the subset $\s\subset \m$ can be taken to be arbitrarily close to the supremum of $\srank(A\Delta)$ over all positive definite diagonal operators $\Delta:\R^m\to \R^m$, albeit with worse dependence on $\e$. However, in~\cite{You14-cube} it was shown how to combine the features of Theorem~\ref{th-vershynin} and Theorem~\ref{th-spielman-srivastava} so as to yield this stronger guarantee with the better  dependence on $\e$ that is asserted in Theorem~\ref{th-spielman-srivastava}. This improvement is important for certain geometric applications~\cite{You14-cube}. The new results that are presented below have this stronger ``weighted" feature, but for the sake of simplicity of the initial discussion in the Introduction  we shall first present all the ensuing statements in their ``unweighted" form that corresponds to the way Theorem~\ref{th-spielman-srivastava} is stated above.

A different proof of Theorem~\ref{th-spielman-srivastava} in the special case $AA^*=I_n$ was found by Marcus, Spielman and Srivastava in~\cite{MSS14}, using their powerful method of interlacing polynomials~\cite{MSS15-ramanujan,MSS15-kadison}. In fact, their forthcoming work~\cite{MSS16} obtains Theorem~\ref{thm:MSS schatten}  below, which yields for the first time a restricted invertibility principle for subsets that can be asymptotically larger than the stable rank, with their size depending on the ratio of the Hilbert--Schmidt norm and the Schatten--von Neumann $4$ norm. This result was announced  by Srivastava in his talk at the conference {\em Banach Spaces: Geometry and Analysis}  (Hebrew University, May 2013), and it is actually a precursor to the outstanding subsequent work~\cite{MSS15-kadison}. Its proof  will appear for the first time in the forthcoming preprint~\cite{MSS16}, but we confirmed with the authors that they obtain  Theorem~\ref{thm:MSS schatten} as stated below.

\begin{theorem}[Marcus--Spielman--Srivastava]\label{thm:MSS schatten} Suppose that $k,m,n\in \N$ and let $A:\R^m\to \R^n$ be a linear operator such that $k<\frac14(\|A\|_{\S_2}/\|A\|_{\S_4})^4$. Define $\e\in (3/4,1)$ by $k=(1-\e)\|A\|_{\S_2}^4/\|A\|_{\S_4}^4$.
Then there exists a subset $\s\subset\m $ with $|\s|=k$ such that
\begin{equation}\label{eq:MSS S4}
\|(AJ_\s)^{-1}\|_{\S_\infty}\le \frac{1}{\sqrt{1-2\sqrt{1-\e}}}\cdot \frac{\sqrt{m}}{\|A\|_{\S_2}}.
\end{equation}
\end{theorem}

Theorem~\ref{thm:MSS schatten} can be much better than the previously known restricted invertibility principles at detecting large well-invertible sub-matrices. To state a concrete example, suppose that the singular values of $A$ are $\ss_1(A)\asymp\sqrt[4]{m}$ and $\ss_2(A)\asymp\ss_3(A)\asymp\ldots\asymp\ss_m(A)=1$. Then Theorem~\ref{th-spielman-srivastava}  yields a subset $\sigma\subset \m$ of size of order $\sqrt{m}$ for which the operator norm of the inverse of $AJ_\s$ is $O(1)$, while Theorem~\ref{thm:MSS schatten} yields such a subset whose size is at least a constant multiple of $m$.

The restriction $k<\frac14(\|A\|_{\S_2}/\|A\|_{\S_4})^4$ in Theorem~\ref{thm:MSS schatten} ensures that $\e>3/4$, so that the quantity appearing under the square root in~\eqref{eq:MSS S4} is positive. Thus, in the statement of Theorem~\ref{thm:MSS schatten}  $k$ cannot be arbitrarily close to the ``modified stable rank" $\|A\|_{\S_2}^4/\|A\|_{\S_4}^4$, but this will be remedied below.

It is important to note that the quantity $\|A\|_{\S_2}^4/\|A\|_{\S_4}^4$ is always at least $\srank(A)$. More generally, given $p\in (2,\infty]$, if we define the $p$-stable rank of $A$ to be the quantity
\begin{equation}\label{eq:def p srank}
\srank_p(A)\eqdef \left(\frac{\|A\|_{\S_2}}{\|A\|_{\S_p}}\right)^{\frac{2p}{p-2}},
\end{equation}
then in particular $\srank_4(A)=\|A\|_{\S_2}^4/\|A\|_{\S_4}^4$ and $\srank_\infty(A)=\srank(A)$. We claim that
\begin{equation}\label{eq:p monotonicity}
p\ge q>2\implies \srank_p(A)\le \srank_q(A),
\end{equation}
Indeed, by direct application of H\"older's inequality we have
$$
\|A\|_{\S_q}\le \|A\|_{\S_2}^{\frac{2(p-q)}{q(p-2)}}\cdot \|A\|_{\S_p}^{\frac{p(q-2)}{q(p-2)}},
$$
which simplifies to give~\eqref{eq:p monotonicity}. The limit as $p\to 2^+$ of $\srank_p(A)$ can be computed explicitly, yielding the quantity below, denoted $\entrank(A)$, which we naturally call the entropic stable rank of $A$.

\begin{multline*}
\entrank(A)\eqdef \lim_{p\to 2^+} \srank_p(A)=\exp\bigg(\log\sum_{j=1}^m \ss_j(A)^2-\frac{2\sum_{j=1}^m\ss_j(A)^2\log \ss_j(A)}{\sum_{j=1}^m \ss_j(A)^2}\bigg)
\\=\exp\bigg(\frac{\trace(A^*A)\log \trace(A^*A)-\trace(A^*A\log(A^*A))}{\trace(A^*A)}\bigg)=\|A\|_{\S_2}^2\prod_{j=1}^m \ss_j(A)^{-\frac{2\ss_j(A)^2}{\|A\|_{\S_2}^2}}.
\end{multline*}

As we shall explain in the next section, here we obtain an improved restricted invertibility theorem that in particular yields a strengthening of Theorem~\ref{thm:MSS schatten} that allows one to make use of the $p$-stable rank of $A$ for every $p>2$, thus producing well-invertible sub-matrices of $A$ of size that can be any integer that is less than the entropic stable rank of $A$.

\subsection{Restricted invertibility in terms of rank} Our main new result is the following theorem.

\begin{theorem}\label{thm:main rank theorem} Suppose that $k,m,n\in \N$. Let $A:\R^m\to \R^n$ be a linear operator with $\rank(A)>k$. Then there exists a subset $\s\subset\m$ with $|\s|=k$ such that
\begin{equation}\label{eq:min over r}
\|(AJ_\s)^{-1}\|_{\S_\infty}\lesssim \min_{r\in \{k+1,\ldots,\rank(A)\}} \sqrt{\frac{mr}{(r-k)\sum_{i=r}^m \ss_i(A)^2}}.
\end{equation}
\end{theorem}

\begin{example}\label{example:harmonic}
To illustrate the relation between Theorem~\ref{th-spielman-srivastava}, Theorem~\ref{thm:MSS schatten} and Theorem~\ref{thm:main rank theorem}, consider a linear operator $A:\R^m\to \R^n$ with $\ss_j(A)\asymp {1/\sqrt{j}}$ for every $j\in \m$. Thus $\rank(A)=m$, $\srank(A)\asymp \log m$ and $\srank_4(m)\asymp (\log m)^2$. Since $\sqrt{m}/\|A\|_{\S_2}\asymp \sqrt{m/\log m}$, Theorem~\ref{th-spielman-srivastava} yields $\s\subset\m$ with $|\s|\asymp \log m$ and $\|(AJ_\s)^{-1}\|_{\S_\infty}\lesssim \sqrt{m/\log m}$, Theorem~\ref{thm:MSS schatten} yields such a subset with $|\s|\asymp (\log m)^2$, and Theorem~\ref{thm:main rank theorem} yields such a subset with $|\s|\gtrsim \sqrt{m}$. In fact, for every $\e\in (0,1)$, Theorem~\ref{thm:main rank theorem} yields  $\s\subset\m$ with $|\s|\gtrsim m^{1-\e}$ such that $\|(AJ_\s)^{-1}\|_{\S_\infty}\lesssim \frac{1}{\sqrt{\e}}\sqrt{m/\log m}$.
\end{example}

Theorem~\ref{thm:main rank theorem}  has the feature that it asserts the existence of a coordinate subspace of dimension arbitrarily close to the rank of the given operator on which it is invertible, with quantitative control on the operator norm of the inverse. The rank is not a stable quantity, but it is simple to deduce stable consequences of Theorem~\ref{thm:main rank theorem}   that are stronger than Theorem~\ref{thm:MSS schatten}. Indeed, continuing with the notation of Theorem~\ref{thm:main rank theorem}, for every $p\in (2,\infty)$ we can apply H\"older's inequality to deduce that
\begin{multline*}
\|A\|_{\S_2}^2=\sum_{i=1}^{r-1}\ss_i(A)^2+\sum_{i=r}^m\ss_i(A)^2\\\le (r-1)^{1-\frac{2}{p}}\bigg(\sum_{i=1}^{r-1}\ss_i(A)^p\bigg)^{\frac{2}{p}}+\sum_{i=r}^m\ss_i(A)^2\le (r-1)^{1-\frac{2}{p}}\|A\|_{\S_p}^2+\sum_{i=r}^m\ss_i(A)^2.
\end{multline*}
Hence,
\begin{equation}\label{eq:ky fan lower}
\sum_{i=r}^m\ss_i(A)^2\ge \|A\|_{\S_2}^2-(r-1)^{1-\frac{2}{p}}\|A\|_{\S_p}^2\stackrel{\eqref{eq:def p srank}}{=}\|A\|_{\S_2}^2\left(1-\bigg(\frac{r-1}{\srank_p(A)}\bigg)^{1-\frac{2}{p}}\right).
\end{equation}
A substitution of~\eqref{eq:ky fan lower} into~\eqref{eq:min over r} yields the following estimate.
\begin{equation}\label{eq:set up to ptimize over r}
\ss_{\min}(AJ_\s)^2\gtrsim \max_{r\in \{k+1,\ldots, \srank_p(A)\}}\left(1-\frac{k}{r}\right)\left(1-\bigg(\frac{r-1}{\srank_p(A)}\bigg)^{1-\frac{2}{p}}\right)\cdot \frac{\|A\|_{\S_2}^2}{m}.
\end{equation}

The estimate~\eqref{eq:set up to ptimize over r} is nontrivial only when $k<\srank_p(A)$, so write $k=(1-\e)\srank_p(A)$ for some $\e\in (0,1)$. One checks that the following choice of $r\in \{k+1,\ldots, \srank_p(A)\}$ attains the maximum in the right hand side of~\eqref{eq:set up to ptimize over r}, up to universal constant factors. If $\e$ is bounded away from $1$, say $\e\in (0,1/2]$, choose $r\asymp (1-\e/2)\srank_p(A)$. If $1/2< \e\le 1-e^{-p/(p-2)}$ then choose $r\asymp \log(1/(1-\e))\cdot \srank_p(A)$. If $1-e^{-p/(p-2)}<\e<1$ then choose $r\asymp e^{-p/(p-2)}\srank_p(A) $. Thus,

\begin{align*}
\begin{split}
0<\e\le \frac12&\implies \|(AJ_\s)^{-1}\|_{\S_\infty}\lesssim \sqrt{\frac{p}{p-2}}\cdot \frac{\sqrt{m}}{\e\|A\|_{\S_2}},\\
\frac12<\e\le 1-e^{-\frac{p}{p-2}}&\implies \|(AJ_\s)^{-1}\|_{\S_\infty}\lesssim \sqrt{\frac{p}{p-2}}\cdot \frac{\sqrt{m}}{\log\left(1/(1-\e)\right)\|A\|_{\S_2}},\\
1-e^{-\frac{p}{p-2}}<\e<1&\implies \|(AJ_\s)^{-1}\|_{\S_\infty}\lesssim \frac{\sqrt{m}}{\|A\|_{\S_2}}.
\end{split}
\end{align*}
A more concise way to write these estimates is as follows.
\begin{equation*}
\|(AJ_\s)^{-1}\|_{\S_\infty}\lesssim \bigg(1+\frac{p}{(p-2)\big|\log\left(1-\e^2\right)\big|}\bigg)^{\frac12} \frac{\sqrt{m}}{\|A\|_{\S_2}}.
\end{equation*}
For ease of future reference, we record the above corollary of Theorem~\ref{thm:main rank theorem}  as Theorem~\ref{thm:our schatten} below.

\begin{theorem}[Restricted invertibility in terms of Schatten--von Neumann norms]\label{thm:our schatten} Suppose that  $k,m,n\in \N$, $\e\in (0,1)$ and $p\in (2,\infty)$. Let $A:\R^m\to \R^n$ be a linear operator that satisfies $k\le (1-\e)\srank_p(A)$. Then there exists a subset $\s\subset \m$ with $|\s|=k$ such that
\begin{equation*}
\|(AJ_\s)^{-1}\|_{\S_\infty}\lesssim \bigg(1+\frac{p}{(p-2)\big|\log\left(1-\e^2\right)\big|}\bigg)^{\frac12} \frac{\sqrt{m}}{\|A\|_{\S_2}}.
\end{equation*}
Equivalently, if $k<\entrank(A)$ then there exists $\sigma \subset \m$ with $|\s|=k$ such that
$$
\|(AJ_\s)^{-1}\|_{\S_\infty}\lesssim \inf_{p>2} \psi_p\!\left(1-\frac{k}{\srank_p(A)}\right)\frac{\sqrt{m}}{\|A\|_{\S_2}},
$$
where $\psi_p:\R\to [0,\infty]$ is defined by $\psi_p(\e)=\infty$ if $\e\le 0$, $\psi_p(x)=(\sqrt{p/(p-2)})/\e$ if $0<\e<1/2$, $\psi_p(\e)=(\sqrt{p/(p-2)})/\log(1/(1-\e))$ if $1/2<\e\le 1-e^{-p/(p-2)}$ and $\psi_p(\e)=1$ if $\e> 1-e^{-p/(p-2)}$.
\end{theorem}

The case $p=4$ of Theorem~\ref{thm:our schatten} implies (up to constant factors) the conclusion of Theorem~\ref{thm:MSS schatten}, though  now treating any $\e\in (0,1)$, i.e., $k$ arbitrarily close to $\srank_4(A)$, while Theorem~\ref{thm:MSS schatten} applies only when $\e>3/4$. Theorem~\ref{thm:our schatten} can detect the well-invertibility of $A$ on coordinate subspaces that are much larger than those detected by Theorem~\ref{thm:MSS schatten}. For example suppose that the singular values of $A$ are $\ss_1(A)\asymp\sqrt[3]{m}$ and $\ss_2(A)\asymp\ss_3(A)\asymp\ldots\asymp\ss_m(A)\asymp1$. Then Theorem~\ref{thm:MSS schatten} yields a subset $\sigma\subset \m$ of size of order $m^{2/3}$ for which the operator norm of the inverse of $AJ_\s$ is $O(1)$, while (the case $p=3$ of) Theorem~\ref{thm:our schatten} yields such a subset whose size is proportional to $m$.

\medskip

We shall prove Theorem~\ref{thm:main rank theorem} through an application of Theorem~\ref{thm:full column rank} below, which is a restricted invertibility statement of independent interest, in combination with a volumetric argument that leads to Lemma~\ref{lem:max volume} below. Throughout what follows, given $n\in \N$ and a linear subspace $F\subset \R^n$, we shall denote the orthogonal projection from $\R^n$ onto $F$ by $\proj_F:\R^n\to F$.

\begin{theorem}\label{thm:full column rank}
Fix $k,m,n\in \N$ and a linear operator $A:\R^m \to \R^n$ satisfying $\rank(A)>k$. Let $\omega\subset \m$ be any subset with $|\omega|=\rank(A)$ such that the vectors $\{Ae_i\}_{i\in \omega}\subset \R^n$ are linearly independent. For every $j\in \omega$ let $F_j\subset \R^n$ be the orthogonal complement of the span of $\{Ae_i\}_{i\in \omega\setminus\{j\}}\subset \R^n$, i.e.,
\begin{equation}\label{eq:def Fj}
F_j\eqdef \left(\spn\left\{A e_i\right\}_{i\in \omega\setminus \{j\}}\right)^\perp.
\end{equation}
Then there exists a  subset $\s\subset \omega$ with $|\s|= k$ such that
\begin{equation}\label{eq:max projection}
\|(AJ_\s)^{-1}\|_{\S_\infty}\lesssim  \frac{\sqrt{\rank(A)}}{\sqrt{\rank(A)-k}}\cdot \max_{j\in \omega} \frac{1}{\|\proj_{F_j}Ae_j\|_2}.
\end{equation}
\end{theorem}

The link between Theorem~\ref{thm:full column rank} and Theorem~\ref{thm:main rank theorem} is furnished through the following lemma.

\begin{lemma}\label{lem:max volume}
Fix $r,m,n\in \N$. Let $A:\R^m\to \R^n$ be a linear operator with $\rank(A)\ge r$. For every $\tau \subset \m$ let $E_\tau\subset \R^n$ be the orthogonal complement of the span of $\{Ae_j\}_{j\in \tau}\subset \R^n$, i.e.,\footnote{Comparing~\eqref{eq:def Fj} and~\eqref{eq:def E sigma} we see that $F_j=E_{\omega\setminus\{j\}}$ for every $j\in \omega$.}
\begin{equation}\label{eq:def E sigma}
E_\tau\eqdef \left(\spn\left\{A e_j\right\}_{j\in \tau}\right)^\perp.
\end{equation}
Then there exists a subset $\tau\subset \m$ with $|\tau|=r$ such that
\begin{equation}\label{eq:desired projections big unweighted}
\forall\, j\in \tau,\qquad \big\|\proj_{E_{\tau\setminus\{j\}}}Ae_j\big\|_2 \ge \frac{1}{\sqrt{m}}\bigg(\sum_{i=r}^m \ss_i(A)^2\bigg)^{\frac12}.
\end{equation}
\end{lemma}

The deduction of  Theorem~\ref{thm:main rank theorem} from Theorem~\ref{thm:full column rank} and Lemma~\ref{lem:max volume} is simple. Indeed, in the setting of Theorem~\ref{thm:main rank theorem}, take $r\in \{k+1,\ldots,\rank(A)\}$ and apply Lemma~\ref{lem:max volume} to obtain a subset $\tau\subset\m$ with $|\tau|=r$ that satisfies~\eqref{eq:desired projections big unweighted}. This implies in particular that $\{Ae_j\}_{j\in \tau}$ are linearly independent, hence the operator $AJ_\tau:\R^\tau\to \R^n$ has rank $r$. By Theorem~\ref{thm:full column rank} applied with $A$ replaced by $AJ_\tau$, $m=r=\rank(A)$ and $\omega=\tau$, we obtain a further subset $\s\subset \tau$ with $|\s|=k$ such that
$$
\|(AJ_\s)^{-1}\|_{\S_\infty}\stackrel{\eqref{eq:max projection}\wedge \eqref{eq:desired projections big unweighted}}{\lesssim} \sqrt{\frac{mr}{(r-k)\sum_{i=r}^m \ss_i(A)^2}}.
$$
This is precisely the assertion of Theorem~\ref{thm:main rank theorem}.

In Section~\ref{sec:MSS} we shall prove the following variant of Theorem~\ref{thm:full column rank}.

\begin{theorem}\label{thm:MSS version}
Fix $k,m,n\in \N$ and a linear operator $A:\R^m \to \R^n$ satisfying $\rank(A)>k$. Then there exists a subset $\s\subset \m$ with $|\s|=k$ such that
\begin{equation}\label{eq:with average projection}
\|(AJ_\s)^{-1}\|_{\S_\infty}\le \frac{\sqrt{m}}{\sqrt{\rank(A)}-\sqrt{k}}\bigg(\frac{1}{\rank(A)}\sum_{i=1}^{\rank(A)}\frac{1}{\ss_i(A)^2}\bigg)^{\frac12}.
\end{equation}
\end{theorem}
To explain how Theorem~\ref{thm:MSS version} relates to Theorem~\ref{thm:main rank theorem}, note that in the setting of Theorem~\ref{thm:main rank theorem} we have
\begin{equation}\label{eq:trace projection identity}
\sum_{j\in \omega} \frac{1}{\|\proj_{F_j}Ae_j\|_2^2}=\sum_{i=1}^{\rank(A)}\frac{1}{\ss_i(AJ_\omega)^2}.
\end{equation}
The simple linear-algebraic justification of~\eqref{eq:trace projection identity} appears in Section~\ref{sec:dual basis} below. For simplicity suppose that $\omega=\m$, so $\rank(A)=m$, and write $k=(1-\e)m$ for some $\e\in (0,1)$. Then Theorem~\ref{thm:main rank theorem} yields a subset $\s\subset\m$ with $|\s|=k$ such that
\begin{equation}\label{incomparable 1}
\|(AJ_\s)^{-1}\|_{\S_\infty}\lesssim  \frac{1}{\sqrt{\e}}\cdot \max_{j\in \m} \frac{1}{\|\proj_{F_j}Ae_j\|_2},
\end{equation}
while, due to~\eqref{eq:trace projection identity}, Theorem~\ref{thm:MSS version} yields a subset $\s\subset\m$ with $|\s|=k$ such that
\begin{equation}\label{incomparable 2}
\|(AJ_\s)^{-1}\|_{\S_\infty}\le  \frac{1}{1-\sqrt{1-\e}}\bigg(\frac{1}{m}\sum_{i=1}^{m}\frac{1}{\|\proj_{F_j}Ae_j\|_2^2}\bigg)^{\frac12}\asymp \frac{1}{\e}\bigg(\frac{1}{m}\sum_{i=1}^{m}\frac{1}{\|\proj_{F_j}Ae_j\|_2^2}\bigg)^{\frac12}.
\end{equation}
The estimates~\eqref{incomparable 1} and~\eqref{incomparable 2} are incomparable since~\eqref{incomparable 1} yields a dependence on $\e$ that is better than that of~\eqref{incomparable 2} as $\e\to 0$, while the bound in~\eqref{incomparable 2} is in terms of the average of the quantities $\{1/\|\proj_{F_j}Ae_j\|_2^2\}_{j=1}^m$ rather than their maximum. It remains an interesting open question  whether  one could obtain a restricted invertibility theorem that combines the best terms in~\eqref{incomparable 1} and~\eqref{incomparable 2}.

\begin{remark}
Theorem~\ref{thm:full column rank} is best possible, up to constant factors. Indeed, fix $k,m\in \N$ with $k<m$ and let $B$ be the $m$ by $m$ matrix all of whose diagonal entries equal $m$ and all of whose off-diagonal entries equal $-1$. Then $B$ is positive definite (diagonal-dominant) and we choose $A=\sqrt{B}$. We are thus in the setting of Theorem~\ref{thm:full column rank} with $m=n=\rank(A)$ and $\omega=\m$. The quantity $1/\|\proj_{F_j}Ae_j\|_2^2$ is equal to the $j$'th diagonal entry of $(A^*A)^{-1}=B^{-1}$; see equation~\eqref{eq:equal to diagonal element} in Section~\ref{sec:dual basis} below for a simple justification of this fact. The matrix $B$ is an invertible circulant matrix, and as such $B^{-1}$ is also a circulant matrix whose diagonal entries equal $2/(m+1)$; see~\cite{Dav79,KS12} for more on the explicit evaluation of basic quantities related to circulant matrices, including their inverses and eigenvalues, which we use here. Therefore $1/\|\proj_{F_j}Ae_j\|_2=\sqrt{2/(m+1)}$ for every $j\in \m$, so that the  right hand side of~\eqref{eq:max projection} equals
$\sqrt{2m/((m+1)(m-k))}\asymp 1/\sqrt{m-k}$. At the same time, take any $\s\subset\m$ with $|\s|=k$. Then $(AJ_\s)^*(AJ_\s)=J_\s^*BJ_\s$ corresponds to a $k$ by $k$ matrix whose diagonal entries equal $m$ and whose off-diagonal entries equal $-1$. This is again a circulant matrix whose eigenvalues equal to $m+1$ with multiplicity $k-1$ and $m+1-k$ with multiplicity $1$. Thus $\ss_1(AJ_\s)=\ldots=\ss_{k-1}(AJ_\s)=\sqrt{m+1}$ and $\ss_k(AJ_\s)=\ss_{\min}(AJ_\s)=1/\|(AJ_\s)^{-1}\|_{\S_\infty}=\sqrt{m+1-k}$. This shows that $\|(AJ_\s)^{-1}\|_{\S_\infty}\asymp 1/\sqrt{m-k}$, so that~\eqref{eq:max projection}  is sharp up to constant factors.
\end{remark}

\subsection{Remarks on the proofs.} The original proof of Bourgain and Tzafriri of Theorem~\ref{th-bourgain-tzafriri}  consists of a beautiful combination of probabilistic, combinatorial and analytic arguments. It proceeds roughly along three steps. Firstly, using random selectors one finds a large collection of columns of $A$ that is ``well separated." In the second step one uses the Sauer--Shelah lemma~\cite{Sau72,She72} to find a further subset of the columns such that the inverse of the restriction of $A$ to this subset, when viewed  as an operator from $\ell_2$ to $\ell_1$, has small norm; the Sauer--Shelah lemma is discussed in Section~\ref{sec:sauer shelah} below, since it plays an important role here as well. The third step of the  Bourgain--Tzafriri proof uses tools from functional analysis, specifically the Little Grothendieck's Inequality~\cite{Gro53} and the Pietsch Domination Theorem~\cite{Pie67}, to control the desired Hilbertian operator norm; these analytic tools are used here as well, and are explained in detail in Section~\ref{sec:gro} and Section~\ref{sec:pie} below.

Vershynin's proof of Theorem~\ref{th-vershynin} uses the Bourgain--Tzafriri restricted invertibility theorem as a ``black box," alongside with (unpublished) work of Kashin and Tzafriri (see Theorem~2.5 in~\cite{Ver01}). A key contribution of Verhynin was the idea to work with the Hilbert--Schmidt norm so as to allow for an iterative argument. As we stated earlier, the proof of Spielman and Srivastava of Theorem~\ref{th-spielman-srivastava}  is entirely different from the previously used methods in this context, relying on the `sparsification method' of Batson--Spielman--Srivastava~\cite{BSS12}. This refreshing approach led to many important developments, and it was subsequently augmented by the powerful `method of interlacing polynomials' of Marcus--Spielman--Srivastava, which they used to prove Theorem~\ref{thm:MSS schatten}, showing that one could use higher Schatten--von Neumann norms to address the restricted invertibility problem.

Our starting point here was the realization that one could use ideas and techniques that predate the works of Vershynin, Spielman--Srivastava and Marcus--Spielman--Srivastava  to obtain asymptotically sharp results such as Theorem~\ref{th-spielman-srivastava},  and even to strengthen the statement in terms of higher Schatten--von Neumann norms that is contained in Theorem~\ref{thm:MSS schatten}. These later results were based on the discovery of powerful new techniques, leading to many additional applications (crowned by the solution of the Kadison--Singer problem~\cite{MSS15-kadison}) that are not covered here, but the present work shows how to apply classical methods to improve over the best known bounds on the restricted invertibility problem. Specifically, we rely on the beautiful work of Giannopoulos~\cite{Gia96}, which treats a seemingly unrelated geometric question (see also~\cite{Gia95}), though it is partially inspired by the work of Bourgain--Tzafriri~\cite{BT87} itself, as well as the works of Bourgain--Szarek~\cite{BS88} and Szarek--Talagrand~\cite{ST89} (see also~\cite{Sza91}). The key step is to use Giannopoulos' clever iterative application of the Sauer--Shelah lemma  (Bourgain--Tzafriri used the Sauer--Shelah lemma  only once in their original argument) in the proof of Theorem~\ref{thm:full column rank}. In fact, one could use a geometric statement of Giannopoulos~\cite{Gia96} as a ``black box" so as to obtain a shorter proof of Theorem~\ref{thm:full column rank}; this is carried out in Section~\ref{sec:geometric} below, but only after we present a self-contained argument in Section~\ref{sec:gia}.

Theorem~\ref{thm:MSS version} is of a different nature, since its proof uses the Marcus--Spielman--Srivastava method of interlacing polynomials. We do not see how to prove it using the classical analytic techniques that are utilized elsewhere in this article, and in fact we do not need it for the applications that are obtained here (as we explained earlier,
Theorem~\ref{thm:MSS version} is incomparable to Theorem~\ref{thm:full column rank}, being weaker in terms of the dependence on certain parameters and stronger in other respects). Nevertheless, Theorem~\ref{thm:MSS version} certainly belongs to the family of restricted invertibility results that we study here.

Among the interesting questions that arise naturally from the present work, we ask whether Theorem~\ref{thm:main rank theorem}, Theorem~\ref{thm:our schatten},   Theorem~\ref{thm:full column rank} and Theorem~\ref{thm:MSS version} can be made to be algorithmic. Our current proofs  do not yield a polynomial time algorithm that finds the desired coordinate subspace, due to various reasons, including (but not limited to) the use of the Sauer--Shelah lemma (in Theorem~\ref{thm:main rank theorem}, Theorem~\ref{thm:our schatten} and   Theorem~\ref{thm:full column rank}) and the use of the method of interlacing polynomials (in Theorem~\ref{thm:MSS version}).


\subsection{Roadmap} While this article is primarily devoted to new results, it also has an expository component due to the fact that we are using tools and ideas from diverse fields, with which some readers may not be familiar. Being very much inspired by Matou\v{s}ek's exceptionally clear style of mathematical exposition, we also made an effort for the ensuing arguments to be self-contained by  including quick explanations of classical results that are being used. It seems  impossible to fully achieve a Matou\v{s}ek-style exposition, but hopefully his influence helped us to make an important area of mathematics and a collection of powerful and versatile tools accessible to a wider audience.

Section~\ref{sec:prem} describes auxiliary statements that will be used in the subsequent proofs. These include classical results of major importance to several fields, and we include brief deductions of what we need so as to make this article self-contained. Section~\ref{sec:volume step} contains the proof of Lemma~\ref{lem:max volume}. A self-contained proof of Theorem~\ref{thm:main rank theorem}, using a clever iterative procedure of Giannopoulus~\cite{Gia96}, appears in Section~\ref{sec:gia}. This is followed by Section~\ref{sec:geometric}, where it is shown that Theorem~\ref{thm:main rank theorem} is equivalent to a geometric theorem of Giannopulos~\cite{Gia96}, thus yielding a shorter (but not self-contained) proof of Theorem~\ref{thm:main rank theorem}. Section~\ref{sec:MSS} contains the proof of Theorem~\ref{thm:MSS version}.

\subsection*{Acknowledgements} We thank Bill Johnson for helpful discussions.  This work was initiated while we were participating in the workshop
{\em Beyond Kadison--Singer: paving and consequences} at the American Institute of Mathematics. We thank the organizers for the excellent working conditions.

\section{Preliminaries}\label{sec:prem}

In this section we shall describe several tools that will be used in the ensuing arguments, and derive certain corollaries of them in forms that will be easy to quote as the need arises later.

\subsection{A bit of linear algebra}\label{sec:dual basis} We shall start with elementary linear algebraic reasoning that clarifies the meaning of some of the quantities that were discussed in the Introduction. In particular, we shall see why the identity~\eqref{eq:trace projection identity} holds true.

We work here in the setting of Theorem~\ref{thm:full column rank}, namely we are given  $k,m,n\in \N$ and a linear operator $A:\R^m \to \R^n$ satisfying $\rank(A)>k$. We are also fixing any subset $\omega\subset \m$ with $|\omega|=\rank(A)$ such that the vectors $\{Ae_i\}_{i\in \omega}\subset \R^n$ are linearly independent. For $j\in \omega$ we consider the linear subspace $F_j\subset \R^n$ that is defined in~\eqref{eq:def Fj}, namely $F_j$ is  the orthogonal complement of the span of $\{Ae_i\}_{i\in \omega\setminus\{j\}}\subset \R^n$. For every $j\in \omega$ define a vector $\mathsf{v}_j\in \R^n$ as follows.
\begin{equation}\label{eq:def vj}
\mathsf{v}_j\eqdef \frac{\proj_{F_j}Ae_j}{\|\proj_{F_j}Ae_j\|_2^2}\in \R^n.
\end{equation}

For every $j\in \omega$, since $I_n-\proj_{F_j}$ is the orthogonal projection onto $\spn(\{Ae_i\}_{i\in \omega\setminus\{j\}})\subset \R^n$, we know that $I_n-\proj_{F_j}Ae_j\in \spn(\{Ae_i\}_{i\in \omega\setminus\{j\}})$. So, $\{\proj_{F_j}Ae_j\}_{j\in \omega}\subset \spn(\{Ae_i\}_{i\in \omega})$, and therefore $\{\mathsf{v}_j\}_{j\in \omega}\subset \spn(\{Ae_i\}_{i\in \omega})$. For $j\in \omega$ we have $\langle \proj_{F_j}Ae_j,Ae_j\rangle = \|\proj_{F_j}Ae_j\|_2^2$, so $\langle \mathsf{v}_j,Ae_j\rangle=1$. Also, because $\proj_{F_j}Ae_j $ is orthogonal to $\{Ae_i\}_{i\in \omega\setminus\{j\}}$, we have $\langle \mathsf{v}_j,Ae_i\rangle=0$ for every $i\in \omega\setminus \{j\}$. Since $\{Ae_i\}_{i\in \omega}$ is a basis of $\spn(\{Ae_i\}_{i\in \omega})$ and $\{\mathsf{v}_j\}_{j\in \omega}\subset \spn(\{Ae_i\}_{i\in \omega})$, this means that $\{\mathsf{v}_j\}_{j\in \omega}$ is the {\em unique} dual basis of $\{Ae_i\}_{i\in \omega}$ in $\spn(\{Ae_i\}_{i\in \omega})$.

The operator $(AJ_\omega)^*(AJ_\omega):\R^\omega\to \R^\omega$ has rank $|\omega|=\rank(A)$, hence it is invertible. For every $j\in \omega$ we may therefore consider the vector
$$
\mathsf{w}_j \eqdef (AJ_\omega)\big((AJ_\omega)^*(AJ_\omega)\big)^{-1}e_j\in \spn(\{Ae_i\}_{i\in \omega}).
$$
Observe that for every $i,j\in \omega$ we have
\begin{multline*}
\langle \mathsf{w}_j,Ae_i\rangle = \left\langle (AJ_\omega)\big((AJ_\omega)^*(AJ_\omega)\big)^{-1}e_j,(AJ_\omega) e_i\right\rangle\\=\left\langle (AJ_\omega)^*(AJ_\omega)\big((AJ_\omega)^*(AJ_\omega)\big)^{-1}e_j, e_i\right\rangle=\langle e_j,e_i\rangle.
\end{multline*}
By the uniqueness of the dual basis of $\{Ae_i\}_{i\in \omega}$ in $\spn(\{Ae_i\}_{i\in \omega})$, we conclude that $\mathsf{v}_j=\mathsf{w}_j$ for every $j\in \omega$. This implies in particular that for every $j\in \omega$ we have
\begin{multline}\label{eq:equal to diagonal element}
\frac{1}{\|\proj_{F_j}Ae_j\|_2^2}=\|\mathsf{v}_j\|_2^2=\langle \mathsf{w}_j,\mathsf{w}_j\rangle =\left\langle(AJ_\omega)\big((AJ_\omega)^*(AJ_\omega)\big)^{-1}e_j,(AJ_\omega)\big((AJ_\omega)^*(AJ_\omega)\big)^{-1}e_j \right\rangle\\
=\left\langle\big((AJ_\omega)^*(AJ_\omega)\big)^{-1}e_j,(AJ_\omega)^*(AJ_\omega)\big((AJ_\omega)^*(AJ_\omega)\big)^{-1}e_j \right\rangle= \left\langle\big((AJ_\omega)^*(AJ_\omega)\big)^{-1}e_j,e_j \right\rangle.
\end{multline}
Consequently,
$$
\sum_{j\in \omega} \frac{1}{\|\proj_{F_j}Ae_j\|_2^2}=\sum_{j\in \omega} \left\langle\big((AJ_\omega)^*(AJ_\omega)\big)^{-1}e_j,e_j \right\rangle=\trace\left(\big((AJ_\omega)^*(AJ_\omega)\big)^{-1}\right)=\sum_{i=1}^{\rank(A)}\frac{1}{\ss_i(AJ_\omega)^2}.
$$
This is precisely the identity~\eqref{eq:trace projection identity}. The above discussion, and in particular the auxiliary vectors~\eqref{eq:def vj} and their properties that were derived above, will play a role in later arguments as well.

\subsection{Grothendieck}\label{sec:gro} We shall use later the following important theorem of Grothendieck~\cite{Gro53}.
\begin{theorem}[Little Grothendieck Inequality]\label{thm:gro} Fix $k,m,n\in \N$. Suppose that $T:\R^m\to \R^n$ is a linear operator. Then for every $x_1,\ldots,x_k\in \R^m$ there exists $i\in \m$ such that
\begin{equation}\label{eq:pi/2}
\sum_{r=1}^k \|Tx_r\|_2^2\le \frac{\pi}{2} \|T\|_{\ell_\infty^m\to \ell_2^n}^2 \sum_{r=1}^k x_{ri}^2.
\end{equation}
Here $\|T\|_{\ell_\infty^m\to \ell_2^n}\eqdef \max_{x\in [-1,1]^m}\|Tx\|_2$ is the operator norm of $T$ when it is viewed as an operator from $\ell_\infty^m$ to $\ell_2^n$,  and $x_{ri}=\langle x_r,e_i\rangle$ is the $i$'th coordinate of $x_r\in \R^m$.
\end{theorem}
To see the significance of Theorem~\ref{thm:gro}, note that the definition of the operator norm of $T$ when it is viewed as an operator from $\ell_\infty^m$ to $\ell_2^n$ is nothing more than the smallest $C\ge 0$ such that for every $x\in \R^m$ there exists $i\in \m$ for which $\|Tx\|_2^2\le C^2x_i^2$. So, the case $k=1$ of~\eqref{eq:pi/2} without the factor $\pi/2$ in the right hand side  is a tautology. Theorem~\eqref{thm:gro} asserts that the case $k=1$ of~\eqref{eq:pi/2} automatically ``upgrades" to~\eqref{eq:pi/2} for general $k\in \N$ at the cost of a loss of the constant factor $\pi/2$.

 The literature contains clear expositions of Theorem~\ref{thm:gro} and its various useful generalizations and equivalent formulations; see e.g.~\cite{Pis86,DJT95}. Nevertheless, for the sake of completeness we shall now quickly explain why Theorem~\ref{thm:gro} holds true, following (a specialization of) the standard proofs of this fact~\cite{Pis86,DJT95}. We note that the factor $\pi/2$ in~\eqref{eq:pi/2} is sharp; see e.g.~the remark immediately following the proof of Theorem~5.4 in~\cite{Pis86}.

To prove Theorem~\ref{thm:gro}, by rescaling both $T$ and $(x_1,\ldots,x_k)$ we may assume without loss of generality that $\|T\|_{\ell_\infty^m\to \ell_2^n}=1$ and $\sum_{r=1}^k\|Tx_r\|_2^2=1$. With this normalization, we claim that
\begin{equation}\label{eq:dual pi/2}
\sum_{j=1}^m \bigg(\sum_{r=1}^k (T^*Tx_{r})_j^2\bigg)^{\frac12}\le \sqrt{\frac{\pi}{2}}.
\end{equation}
Once proven, \eqref{eq:dual pi/2} implies the desired estimate~\eqref{eq:pi/2} via the following application of Cauchy--Schwarz.
\begin{multline*}
1=\sum_{r=1}^k\|Tx_r\|_2^2=\sum_{r=1}^k \langle x_r,T^*Tx_r\rangle=\sum_{j=1}^m \sum_{r=1}^k x_{rj}(T^*Tx_r)_j\le \sum_{j=1}^m \bigg(\sum_{r=1}^k x_{rj}^2\bigg)^{\frac12}\bigg(\sum_{r=1}^k (T^*Tx_{r})_j^2\bigg)^{\frac12}\\\le
\max_{i\in \m} \bigg(\sum_{r=1}^k x_{ri}^2\bigg)^{\frac12} \sum_{j=1}^m \bigg(\sum_{r=1}^k (T^*Tx_{r})_j^2\bigg)^{\frac12}\stackrel{\eqref{eq:dual pi/2}}{\le}
\sqrt{\frac{\pi}{2}}\cdot \max_{i\in \m} \bigg(\sum_{r=1}^k x_{ri}^2\bigg)^{\frac12} .
\end{multline*}

To prove~\eqref{eq:dual pi/2}, let $\{g_r\}_{r=1}^k$ be i.i.d. standard Gaussian random variables. For every $j\in \m$ the random variable $\sum_{r=1}^k g_r (T^*Tx_r)_j$ is Gaussian with mean $0$ and variance $\sum_{r=1}^k (T^*Tx_r)_j^2$. So,
\begin{multline}\label{eq:use rotation invariance}
\E\bigg[\sum_{j=1}^m\Big|\Big(T^*\sum_{r=1}^k g_r Tx_r\Big)_j\Big|\bigg]=\E\bigg[\sum_{j=1}^m\Big|\sum_{r=1}^k g_r (T^*Tx_r)_j\Big|\bigg]=\sum_{j=1}^m\E\bigg[\Big|\sum_{r=1}^k g_r (T^*Tx_r)_j\Big|\bigg]\\=\E\big[|g_1|\big]\sum_{j=1}^m \bigg(\sum_{r=1}^k(T^*Tx_r)_j^2\bigg)^{\frac12}=\sqrt{\frac{2}{\pi}}\sum_{j=1}^m \bigg(\sum_{r=1}^k(T^*Tx_r)_j^2\bigg)^{\frac12}.
\end{multline}
Let $z\in \{-1,1\}^m$ be the random vector given by $z_j\eqdef \sign \Big(\big(T^*\sum_{r=1}^k g_r Tx_r\big)_j\Big)$. Then
\begin{multline}\label{eq:to take gaussian expectation}
\sum_{j=1}^m\Big|\Big(T^*\sum_{r=1}^k g_r Tx_r\Big)_j\Big|=\bigg\langle z,T^*\sum_{r=1}^k g_r Tx_r\bigg\rangle=\bigg\langle Tz,\sum_{r=1}^k g_r Tx_r\bigg\rangle\\\le \|Tz\|_2\cdot  \Big\|\sum_{r=1}^k g_r Tx_r\Big\|_2\le \|T\|_{\ell_\infty^m\to \ell_2^n}\cdot \|z\|_\infty\cdot \Big\|\sum_{r=1}^k g_r Tx_r\Big\|_2=\Big\|\sum_{r=1}^k g_r Tx_r\Big\|_2.
\end{multline}
 By taking expectations in~\eqref{eq:to take gaussian expectation} we see that
 \begin{multline*}
\sqrt{\frac{2}{\pi}}\sum_{j=1}^m \bigg(\sum_{r=1}^k(T^*Tx_r)_j^2\bigg)^{\frac12} \stackrel{\eqref{eq:use rotation invariance}}{=}\E\bigg[\sum_{j=1}^m\Big|\Big(T^*\sum_{r=1}^k g_r Tx_r\Big)_j\Big|\bigg]\\\stackrel{\eqref{eq:to take gaussian expectation}}{\le} \E\bigg[\Big\|\sum_{r=1}^k g_r Tx_r\Big\|_2\bigg]\le \bigg(\E\bigg[\Big\|\sum_{r=1}^k g_r Tx_r\Big\|_2^2\bigg]\bigg)^{\frac12}=\sum_{r=1}^k\|Tx_r\|_2^2=1,
\end{multline*}
 This is precisely the desired estimate~\eqref{eq:dual pi/2}, thus completing the proof of Theorem~\ref{thm:gro}. \qed

\subsection{Pietsch}\label{sec:pie} Another classical tool that will be used later (together with the Little Grothendieck Inequality) is the Pietsch Domination Theorem~\cite{Pie67}.

\begin{theorem}[Pietsch Domination]\label{thm:pie}
Fix $m,n\in \N$ and $M\in (0,\infty)$. Suppose that $T:\R^m\to \R^n$ is a linear operator such that for every $k\in \N$ and $x_1,\ldots,x_k\in \R^m$ there exists $i\in \m$ with $\sum_{r=1}^k \|Tx_r\|_2^2\le M^2\sum_{r=1}^k x_{ri}^2$. Then there exist $\mu_1,\ldots,\mu_m\in [0,1]$ with $\sum_{i=1}^m\mu_i=1$ such that
\begin{equation*}\label{eq:M summing norm}
\forall\, w=(w_1,\ldots,w_m)\in \R^m,\qquad \|Tw\|_2^2\le M^2\sum_{i=1}^m \mu_iw_i^2.
\end{equation*}
\end{theorem}
 Observe in passing that the conclusion of Theorem~\ref{thm:pie} immediately implies its assumption. Indeed, by applying this conclusion with $w=x_r$ for each $r\in \{1,\ldots,k\}$, and then summing the resulting inequalities over $r\in \{1,\ldots,k\}$, we get that  $\sum_{r=1}^k \|Tx_r\|_2^2\le \sum_{i=1}^m\mu_i(M^2\sum_{r=1}^k x_{ri}^2)$, so the existence of the desired index $i\in \m$ follows from the fact that $(\mu_1,\ldots,\mu_m)$ is a probability measure. The main point here is therefore the reverse implication, as stated in Theorem~\ref{thm:pie}.

In Banach space theoretic terminology, the assumption on the operator $T$ in Theorem~\ref{thm:pie} says that $T$ has {\em $2$-summing norm at most} $M$ when it is viewed as an operator from $\ell_\infty^m$ to $\ell_2^n$. We refer to the monographs~\cite{Tom89,DJT95} for much more on this topic, as well as proofs of (more general versions of) the Pietsch Domination Theorem. As before, for the sake of completeness we shall now explain why Theorem~\ref{thm:pie} holds true, following (a specialization of) the standard proofs~\cite{Tom89,DJT95} of this fact, which amount to an application of the separation theorem (equivalently, Hahn--Banach or duality of linear programming) to appropriately chosen convex sets.

Let $\mathsf{K}\subset \R^m$ be the set of all those vectors $y\in \R^m$ for which there exists $k\in \N$ and $x_1,\ldots,x_k\in \R^m$ such that $y_i=\sum_{r=1}^k \|Tx_r\|_2^2-M^2\sum_{r=1}^k x_{ri}^2$ for every $i\in \m$. It is immediate to check that  $\mathsf{K}$ is convex, and the assumption on $T$ can be restated as saying that $\mathsf{K}\cap (0,\infty)^m=\emptyset$. By the separation theorem there exist $\mu=(\mu_1,\ldots,\mu_m)\in \R^m$ such that $\sum_{i=1}^m\mu_iy_i<\sum_{i=1}^m \mu_i z_i$ for every $y\in \mathsf{K}$ and $z\in (0,\infty)^m$. In particular, $\mu\neq 0$ and $\inf_{z\in (0,\infty)^m} \langle z,\mu\rangle>-\infty$, so necessarily  $\mu_i\ge 0$ for all $i\in \m$. We may rescale so that $\sum_{i=1}^m \mu_i=1$. If $w\in \R^m$ then  $(\|Tw\|_2^2-M^2w_i^2)_{i=1}^m\in \mathsf{K}$, so $\|Tw\|_2^2-M^2\sum_{i=1}^m\mu_iw_i^2=\sum_{i=1}^m \mu_i(\|Tw\|_2^2-M^2w_i^2)\le \inf_{z\in (0,\infty)^m} \sum_{i=1}^m\mu_iz_i=0$.\qed

\medskip
The following lemma is a combination of the Little Grothendieck Inequality and the Pietsch Domination Theorem; this is how Theorem~\ref{thm:gro} and Theorem~\ref{thm:pie} will be used in what follows.

\begin{lemma}\label{lem:grothendieck pietsch}
Fix $m,n\in \N$ and $\e\in (0,1)$. Let $T:\R^n\to \R^m$ be a linear operator. Then there exists a subset $\s\subset \m$ with $|\s|\ge (1-\e)m$ such that
\begin{equation}\label{eq:norm improvement with proj}
\left\|\proj_{\R^\s}T\right\|_{\S_\infty}\le \sqrt{\frac{\pi}{2\e m}}\cdot  \|T\|_{\ell_2^n\to \ell_1^m}.
\end{equation}
\end{lemma}

\begin{proof}
Since we have $\|T^*\|_{\ell_\infty^m\to \ell_2^n}=\|T\|_{\ell_2^n\to \ell_1^m}$, an application of Theorem~\ref{thm:gro} to $T^*:\R^m\to \R^n$ shows that the assumption of Theorem~\ref{thm:pie} holds true with $T$ replaced by $T^*$ and $M=\sqrt{\pi/2}\cdot \|T\|_{\ell_2^n\to \ell_1^m}$. Hence, Theorem~\ref{thm:pie} shows that there exists $\mu \in [0,1]^m$ with $\sum_{i=1}^m\mu_i=1$ such that
\begin{equation}\label{eq:T* norm}
\forall\, y\in \R^m,\qquad \|T^*y\|_2^2\le \frac{\pi}{2}\|T\|_{\ell_2^n\to \ell_1^m}^2\sum_{i=1}^m\mu_iy_i^2.
\end{equation}
Define
\begin{equation}\label{eq:sigma from domination measure}
\s\eqdef \left\{i\in \m:\ \mu_i\le \frac{1}{m\e}\right\}.
\end{equation}
Since $\mu$ is a probability measure on $\m$, by Markov's inequality we have $|\s|\ge (1-\e)m$.

Take $x\in \R^n$ and choose $y\in \R^m$ such that $\|y\|_2=1$ and $\|\proj_{\R^\s}Tx\|_2=\langle y,\proj_{\R^\s}Tx\rangle$. Then,
\begin{multline}\label{eq:use markov}
\|\proj_{\R^\s}Tx\|_2^2=\langle y,\proj_{\R^\s}Tx\rangle^2=\langle T^*\proj_{\R^\s}y,x\rangle^2\le \|T^*\proj_{\R^\s}y\|_2^2\cdot\|x\|_2^2\\\stackrel{\eqref{eq:T* norm}}{\le}
\frac{\pi}{2}\|T\|_{\ell_2^n\to \ell_1^m}^2\cdot\|x\|_2^2\sum_{i\in \s}\mu_iy_i^2\stackrel{\eqref{eq:sigma from domination measure}}{\le}\frac{\pi}{2m\e}\|T\|_{\ell_2^n\to \ell_1^m}^2\cdot\|x\|_2^2\cdot\|y\|_2^2=\frac{\pi}{2m\e}\|T\|_{\ell_2^n\to \ell_1^m}^2\cdot\|x\|_2^2.
\end{multline}
Since~\eqref{eq:use markov} holds true for every $x\in \R^n$, this completes the proof of the desired estimate~\eqref{eq:norm improvement with proj}.
\end{proof}

\subsection{Sauer--Shelah}\label{sec:sauer shelah} The Sauer--Shelah lemma~\cite{Sau72,She72} is a fundamental combinatorial principle of wide applicability that will be used crucially later.

\begin{lemma}[Sauer--Shelah]\label{sauer-shelah}
Fix $m,n\in \N$. Suppose that $\Omega\subseteq \{-1,1\}^n$ satisfies
$
|\Omega| > \sum_{k=0}^{m-1} {n\choose k}.
$
Then there exists a subset $\sigma\subseteq \n$ with  $|\s|\ge m$ such that $\proj_{\R^\sigma} \Omega= \{-1,1\}^\s$, i.e., for every $\e\in \{-1,1\}^\s$ there exists $\d\in \Omega$ such that $\d_j=\e_j$ for every $j\in \s$. In particular, if $|\Omega|> 2^{n-1}$ then such a subset $\s\subset \n$ exists with $|\s|\ge \lceil (n+1)/2\rceil \ge n/2$.
\end{lemma}
It is simple to prove Lemma~\ref{sauer-shelah} by induction on $n$ when one strengthens the inductive hypothesis as follows. Denoting $\mathrm{\bf sh}(\Omega)= \{\s\subset \n:\ \proj_{\R^\sigma} \Omega= \{-1,1\}^\s\}$, we claim that $|\mathrm{\bf sh}(\Omega)|\ge |\Omega|$; this would imply Lemma~\ref{sauer-shelah} since the number of subsets of $\n$ of size at most $m-1$ equals $\sum_{k=0}^{m-1} {n\choose k}$. This stronger statement is due to Pajor~\cite{Paj85}, and the resulting very short inductive proof which we shall now sketch for completeness appears as Theorem~1.1 in~\cite{ARS02}.

The case $n=1$ holds trivially (here we use the convention that $\{-1,1\}^\emptyset =\emptyset $ and $\proj_{\R^\emptyset}\Omega=\emptyset$). Assuming the validity of the above statement for $n$, take $\Omega\subset\{-1,1\}^{n+1}=\{-1,1\}^n\times \{-1,1\}$ and denote $\Omega_{1}=\{x\in \{-1,1\}^n:\ (x,1)\in \Omega\}$ and $\Omega_{-1}=\{x\in \{-1,1\}^n:\ (x,-1)\in \Omega\}$. Then $|\Omega_1|+|\Omega_{-1}|=|\Omega|$ and by the inductive hypothesis we have $|\mathrm{\bf sh}(\Omega_1)|\ge |\Omega_1|$ and $|\mathrm{\bf sh}(\Omega_{-1})|\ge |\Omega_{-1}|$. By our definitions we have $\mathrm{\bf sh}(\Omega)\supseteq (\mathrm{\bf sh}(\Omega_1)\cup \mathrm{\bf sh}(\Omega_{-1}))\cupdot \{\s\cup \{n+1\}:\ \s\in \mathrm{\bf sh}(\Omega_1)\cap \mathrm{\bf sh}(\Omega_{-1})\}$, so  $|\mathrm{\bf sh}(\Omega)|\ge |\mathrm{\bf sh}(\Omega_1)\cup \mathrm{\bf sh}(\Omega_{-1})|+|\mathrm{\bf sh}(\Omega_1)\cap \mathrm{\bf sh}(\Omega_{-1})|=|\mathrm{\bf sh}(\Omega_1)|+ |\mathrm{\bf sh}(\Omega_{-1})|\ge |\Omega_1|+|\Omega_{-1}|=|\Omega|$. \qed

\subsection{Fan and Hilbert--Schmidt}
We record for ease of future use the following lemma that controls the influence of multiplication by an orthogonal projection on the Hilbert--Schmidt norm of a linear operator. Its proof is a simple consequence of the classical {\em Fan Maximum Principle}~\cite{Fan49}, but we couldn't locate a reference where it is stated explicitly in the form that we will use later.

\begin{lemma}\label{lem:ky}
Fix $m,n\in \N$ and $r\in \n$. Let $A:\R^m\to \R^n$ be a linear operator and let $\P:\R^n\to \R^n$ be an orthogonal projection of rank $r$. Then
$$
\|\P A\|_{\S_2}\ge \bigg(\sum_{i=n-r+1}^m \ss_i(A)^2\bigg)^{\frac12}.
$$
\end{lemma}

\begin{proof} Since $I_n-\P$ is an orthogonal projection of rank $n-r$, by a classical result of Fan~\cite{Fan49},
\begin{equation}\label{eq:use ky fan}
\trace(AA^*(I_n-\P))\le \sum_{i=1}^{n-r} \ss_i(AA^*)= \sum_{i=1}^{n-r} \ss_i(A)^2
\end{equation}
The proof of~\eqref{eq:use ky fan} is simple; see e.g.~\cite[Lemma~8.1.8]{Sto13} for a short proof and~\cite[Chapter~III]{Bha97} for more general variational principles along these lines. Now, since $\P$ is an orthogonal projection,
\begin{multline*}
\|\P A\|_{\S_2}^2 =\trace((\P A)^*(\P A))=\trace (A^*\P A)=\trace(AA^* \P)=\trace(AA^*)-\trace(AA^*(I_n-\P))\\=\sum_{i=1}^m \ss_i(A)^2-\trace(AA^*(I_n-\P))\stackrel{\eqref{eq:use ky fan}}{\ge} \sum_{i=1}^m \ss_i(A)^2-\sum_{i=1}^{n-r} \ss_i(A)^2=\sum_{i=n-r+1}^{m} \ss_i(A)^2.\tag*{\qedhere}
\end{multline*}
\end{proof}

\section{Proof of Lemma~\ref{lem:max volume}}\label{sec:volume step}

In this section we shall prove Lemma~\ref{lem:max volume} in a more general weighted form that corresponds to the renormalization step in Vershynin's Theorem, i.e., Theorem~\ref{th-vershynin}. Using this weighted version of Lemma~\ref{lem:max volume}, one can directly deduce weighted versions of Theorem~\ref{thm:main rank theorem} and Theorem~\ref{thm:our schatten} as well, by combining Lemma~\ref{lem:extract big projections} below  with Theorem~\ref{thm:full column rank}, exactly as we did in the Introduction.

\begin{lemma}[weighted version of Lemma~\ref{lem:max volume}]\label{lem:extract big projections} Fix $r,m,n\in \N$. Let $A:\R^m\to \R^n$ be a linear operator with $\rank(A)\ge r$. For every $\tau \subset \m$ let $E_\tau\subset \R^n$ be defined as in~\eqref{eq:def E sigma}, i.e., it is the orthogonal complement of the span of $\{Ae_j\}_{j\in \tau}\subset \R^n$. Then for every $d_1,\ldots,d_m\in (0,\infty)$ there exists a subset $\tau\subset \m$ with $|\tau|=r$ such that
\begin{equation}\label{eq:desired projections big}
\forall\, j\in \tau,\qquad \big\|\proj_{E_{\tau\setminus\{j\}}}Ae_j\big\|_2 \ge \frac{d_j}{\sqrt{\sum_{i=1}^m d_i^2}} \bigg(\sum_{i=r}^m \ss_i(A)^2\bigg)^{\frac12}.
\end{equation}
\end{lemma}

\begin{proof} For every $\tau\subset \m$ let $K_\tau\subset \R^n$ be the convex hull of the vectors $\{\pm Ae_j/d_j\}_{j\in \tau}$, i.e.,
\begin{equation}\label{eq:def K sigma}
K_\tau\eqdef \conv\left(\left\{\frac{1}{d_j}Ae_j:\ j\in \tau\right\}\cup \left\{-\frac{1}{d_j}Ae_j:\ j\in \tau\right\}\right).
\end{equation}
The desired subset $\tau\subset\m$ will be chosen so as to maximize the $r$-dimensional volume of the convex hull of $K_\s$ over all those subsets $\s$ of $\m$ of size $r$. Namely, we shall fix from now on a subset $\tau\subset \m$ with $|\tau|=r$ such that
\begin{equation}\label{eq:sigma choice volume}
\vol_r(K_\tau)=\max_{\substack{\s\subset \m\\ |\s|=r}}\vol_r(K_\s).
\end{equation}

Take any $\beta\subset \m$ with $|\beta|=r-1$ and fix $i\in \m\setminus \beta$. Then by the definition~\eqref{eq:def K sigma} we have $K_{\beta\cup\{i\}}=\conv(\{\pm Ae_i/d_i\}\cup K_\beta)$, i.e.,  $K_{\beta\cup\{i\}}$ is the union of the two cones with base $K_\beta$ and apexes at $\pm Ae_i/d_i$. Recalling~\eqref{eq:def E sigma}, note that $K_\beta\subset \spn(K_\beta)=E_\beta^\perp$. Hence, the   height of these two cones equals the Euclidean length of the orthogonal projection of $Ae_i/d_i$ onto $E_\beta$. Therefore,
\begin{equation}\label{eq:volume formula}
\vol_r\!\left(K_{\beta\cup\{i\}}\right)=\frac{2\big\|\proj_{E_\beta}Ae_i\big\|_2\vol_{r-1}(K_\beta)}{rd_i}.
\end{equation}

Returning to the subset $\tau$ that was chosen in~\eqref{eq:sigma choice volume}, we see that if $j\in \tau$ and $i\in \m$  then
\begin{multline}\label{eq:proj max}
\frac{2\big\|\proj_{E_{\tau\setminus\{j\}}}Ae_j\big\|_2\vol_{r-1}\!\left(K_{\tau\setminus\{j\}}\right)}{rd_j}\stackrel{\eqref{eq:volume formula}}{=}\vol_r(K_\tau)\\\stackrel{\eqref{eq:sigma choice volume}}{\ge} \vol_r\!\left(K_{(\tau\setminus\{j\})\cup\{i\}}\right)\stackrel{\eqref{eq:volume formula}}{=}\frac{2\big\|\proj_{E_{\tau\setminus\{j\}}}Ae_i\big\|_2\vol_{r-1}\!\left(K_{\tau\setminus\{j\}}\right)}{rd_i}.
\end{multline}
Since we are assuming that $r\le \rank(A)$, we know that $\vol_r(K_\tau)>0$. It therefore follows from~\eqref{eq:proj max} that also $\vol_{r-1}\!\left(K_{\tau\setminus\{j\}}\right)>0$, so me may cancel the quantity $2\vol_{r-1}\!\left(K_{\tau\setminus\{j\}}\right)/r$ from both sides of~\eqref{eq:proj max}. Since the resulting estimate holds true for every $i\in \m$, we conclude that
\begin{equation}\label{eq:proj max 2}
\forall\, j\in \tau,\qquad \frac{\big\|\proj_{E_{\tau\setminus\{j\}}}Ae_j\big\|_2}{d_j}=\max_{i\in \m} \frac{\big\|\proj_{E_{\tau\setminus\{j\}}}Ae_i\big\|_2}{d_i}.
\end{equation}
Consequently, for every $j\in \tau$ we have
$$
\frac{\big\|\proj_{E_{\tau\setminus\{j\}}}Ae_j\big\|_2^2}{d_j^2}\bigg( \sum_{i=1}^m d_i^2\bigg)\stackrel{\eqref{eq:proj max 2}}{\ge} \sum_{i=1}^m  \big\|\proj_{E_{\tau\setminus\{j\}}}Ae_i\big\|_2^2=\big\|\proj_{E_{\tau\setminus\{j\}}}A\big\|_{\S_2}^2.
$$
Equivalently,
\begin{equation}\label{eq:projections on both sides}
\forall\, j\in \tau,\qquad \big\|\proj_{E_{\tau\setminus\{j\}}}Ae_j\big\|_2\ge \frac{d_j}{\sqrt{\sum_{i=1}^m d_i^2}} \big\|\proj_{E_{\tau\setminus\{j\}}}A\big\|_{\S_2}.
\end{equation}
Recalling~\eqref{eq:def E sigma}, since $|\tau|=r$ we know that $\dim( E_{\tau\setminus\{j\}})=n-(r-1)$ for every $j\in \tau$. Consequently, $\proj_{E_{\tau\setminus\{j\}}}:\R^n\to \R^n$ is an orthogonal projection of rank $n-(r-1)$, so that the desired inequality~\eqref{eq:desired projections big} follows from~\eqref{eq:projections on both sides} and Lemma~\ref{lem:ky}.
\end{proof}

\section{Giannopoulos}\label{sec:gia}

In this section we shall prove Theorem~\ref{thm:full column rank}, following the lines of a clever iterative procedure that was devised by Giannopoulos in~\cite{Gia96}. Throughout the ensuing discussion, we may assume in the setting of Theorem~\ref{thm:full column rank} that $\omega=\m$, in which case $\rank(A)=m$. Indeed, there is no loss of generality by doing so because for general $\omega\subset \m$ we could then consider the restricted operator $AJ_\omega:\R^\omega\to \R^n$ in order to obtain  Theorem~\ref{thm:full column rank} as stated in the Introduction.

\subsubsection*{Proof overview} The overall strategy of the ensuing proof can be explained in broad strokes given the tools that were already presented in Section~\ref{sec:prem}. The ultimate goal of Theorem~\ref{thm:full column rank} is to obtain an upper bound on the operator norm $\|\cdot \|_{\S_\infty}$ of a certain $m$ by $n$ matrix (the inverse of an appropriate coordinate restriction of the given $n$ by $m$ matrix $A$), while we have already seen in Lemma~\ref{lem:grothendieck pietsch} that if one does not mind composing with a further coordinate projection then such a bound follows automatically from a weaker upper estimate on the operator norm $\|\cdot\|_{\ell_2^n\to \ell_1^m}$. The latter quantity can be controlled using the Sauer--Shelah lemma due to the following reasoning.

Let $\{\mathsf{v}_j\}_{j=1}^m$ be the dual basis of $\{Ae_j\}_{j=1}^m$ that is given in~\eqref{eq:def vj}. Consider the subset $\Omega$ of the hypercube $\{-1,1\}^m$ consisting of all those sign vectors $\e=(\e_1,\ldots,\e_m)$ for which the Euclidean norm $\|\sum_{j=1}^m \e_j \mathsf{v}_j\|_2$ is not too large, with the precise meaning of ``not too large" here to be specified in the proof of  Lemma~\ref{lem:induction} below; see~\eqref{eq:def Omega}. The parallelogram identity says that if $\e\in \{-1,1\}^m$ is chosen uniformly at random then the expectation of $\|\sum_{j=1}^m \e_j \mathsf{v}_j\|_2^2$ equals $\sum_{j=1}^m \|\mathsf{v}_j\|_2^2$. So, by Markov's inequality, an appropriate setting of the parameters would yield that the cardinality of $\Omega$ is greater than $2^{m-1}=|\{-1,1\}^m|/2$. The Sauer--Shelah lemma would then furnish a coordinate subset $\beta\subset \m$ with the property that every sign pattern $(\e_j)_{j\in \beta}\in \{-1,1\}^\beta$ can be completed to a full dimensional sign vector $\e\in \{-1,1\}^m$ such that $\sum_{j=1}^m \e_j \mathsf{v}_j$ is ``short" in the Euclidean norm.

The above conclusion implies an upper bound on the operator norm of the inverse of the restriction of $A$ to $\R^\beta$, when it is viewed as an operator from $\ell_2^\beta$ to $\ell_1^m$. Indeed, given an arbitrary vector $(a_j)_{j\in \beta}\in \R^\beta$, the goal is to bound $\sum_{j\in \beta}|a_j|$ in terms of $\|\sum_{j\in \beta} a_j Ae_j\|_2$. The sign pattern to be considered is then the signs of the coefficients $(a_j)_{j\in \beta}\in \R^\beta$, i.e., set $\e_j=\sign(a_j)$ for every $j\in \beta$. The (Sauer--Shelah) subset $\beta\subset \m$ was constructed so that this sign vector can be completed to a full dimensional sign vector $\e\in \{-1,1\}^m$ with control on the Euclidean length of $\sum_{j=1}^m \e_j \mathsf{v}_j$. But $\{\mathsf{v}_j\}_{j=1}^m$ is a dual basis of $\{Ae_j\}_{j=1}^m$, so by the definition of $(\e_j)_{j\in \beta}$ the quantity $\sum_{j\in \beta}|a_j|$ is equal to the scalar product of $\sum_{j\in \beta} a_j Ae_j$ with the ``short" vector $\sum_{j=1}^m \e_j \mathsf{v}_j$. By Cauchy--Schwarz this scalar product is bounded from above by the Euclidean length of $\sum_{j\in \beta} a_j Ae_j$ times the Euclidean length of $\sum_{j=1}^m \e_j \mathsf{v}_j$, with the latter quantity being bounded above by design.

By Lemma~\ref{lem:grothendieck pietsch} we can now pass to a further subset of $\beta$ and compose the resulting inverse matrix with the coordinate projection onto that subset so as to ``upgrade" this control on the operator norm from $\ell_2^{\beta}$ to $\ell_1^m$ to a better upper bound on $\|\cdot \|_{\S_\infty}$. Complications arise when one examines the above strategy from the quantitative perspective. The Sauer--Shelah lemma can at best produce a coordinate subset of size $m/2$, while we desire to obtain restricted invertibility on a potentially larger subset. Moreover, in the above procedure the Sauer--Shelah subset is further  reduced in size due to the subsequent use of Lemma~\ref{lem:grothendieck pietsch}. Since we desire to extract larger coordinate subsets, one can attempt to apply this reasoning iteratively, i.e., start by using the Sauer--Shelah lemma to obtain a coordinate subset, followed by an application of  Lemma~\ref{lem:grothendieck pietsch} to pass to a further subset $\beta'\subset \m$. Now apply the same double selection procedure to $\m\setminus \beta'$, thus obtaining a subset $\beta''\subset \m\setminus \beta'$, and iterate this procedure by now considering $\m\setminus (\beta'\cup\beta'')$ and so forth. To make this strategy work, one needs to formulate a stronger inductive hypothesis so as to allow one to ``glue" the local information on the subsets that are extracted in each step of the iteration  into global information on their union, while ensuring that the end result is a sufficiently large coordinate subset. This is the reason why the assumptions of Lemma~\ref{lem:induction} below are more complicated. The technical details that implement the above strategy are explained in the remainder of this section.

\begin{lemma}\label{lem:induction}
Fix $n\in \N$ and $m\in \n$. Let $A:\R^m \to \R^n$ be a linear operator such that the vectors $\{Ae_j\}_{j=1}^m\subset \R^n$ are linearly independent. Suppose that $k\in \N\cup\{0\}$ and $\s\subset \m$. For $j\in \m$ recall the definition of the subspace $F_j\subset \R^n$ in~\eqref{eq:def Fj} (with $\omega=\m$), i.e,
\begin{equation*}
F_j=\left(\spn\left\{A e_i\right\}_{i\in \m\setminus \{j\}}\right)^\perp.
\end{equation*}
Then there exists $\tau\subset \s$ with $|\tau|\ge (1-2^{-k})|\s|$ such that for every $\vartheta\subset \m$ that satisfies $\vartheta\supseteq \tau$ and every $a=(a_1,\ldots,a_m)\in \R^m$ there exists an index $j\in \m$ for which
\begin{equation}\label{eq:inductive statement}
\sum_{i\in \tau} |a_i|\le \frac{\sqrt{|\s|}\sum_{r=1}^{k}2^{\frac{r}{2}}}{\|\proj_{F_j}Ae_j\|_2}\bigg\|\sum_{i\in \vartheta}a_i Ae_i\bigg\|_2+(2^k-1)\sum_{i\in \vartheta\cap(\s\setminus\tau)}|a_i|.
\end{equation}
\end{lemma}

\begin{proof}
It will be convenient to introduce the following notation.
\begin{equation}\label{eq:M alpha}
M\eqdef \max_{j\in \m} \frac{1}{\|\proj_{F_j}Ae_j\|_2}\qquad\mathrm{and}\qquad \alpha_k\eqdef \sum_{r=1}^{k}2^{\frac{r}{2}}.
\end{equation}
Throughout we adhere to the convention that an empty sum vanishes, thus in particular $\alpha_0=0$.

Under the notation~\eqref{eq:M alpha}, our goal becomes to show that there exists $\tau\subset \s$ with $|\tau|\ge (1-2^{-k})|\s|$ such that for every $\vartheta\subset \m$ that satisfies $\vartheta\supseteq \tau$ and every $a\in \R^m$ we have
\begin{equation}\label{eq:inductive statement shorter}
\sum_{i\in \tau} |a_i|\le \alpha_kM\sqrt{|\s|}\bigg\|\sum_{i\in \vartheta}a_i Ae_i\bigg\|_2+(2^k-1)\sum_{i\in \vartheta\cap(\s\setminus\tau)}|a_i|.
\end{equation}
We shall prove this statement by induction on $k$. The case $k=0$ holds vacuously by taking $\tau=\emptyset$. Assuming the validity of this statement for $k$, we shall proceed to deduce its validity for $k+1$.

We are given $\tau\subset \s$ with $|\tau|\ge (1-2^{-k})|\s|$ such that for every $\vartheta\subset \m$ that satisfies $\vartheta\supseteq \tau$ we know that~\eqref{eq:inductive statement shorter} holds true for every $a\in \R^m$. Observe that if $\tau=\s$ then $\tau$ itself would satisfy the required statement for $k+1$, so we may assume from now on that $\s\setminus\tau\neq \emptyset$.

For every $j\in \m$ let $\mathsf{v}_j$ be given as in~\eqref{eq:def vj}, i.e.,
\begin{equation}\label{eq:def vj-again}
\mathsf{v}_j\eqdef \frac{\proj_{F_j}Ae_j}{\|\proj_{F_j}Ae_j\|_2^2}\in \R^n.
\end{equation}
Observe that the denominator in~\eqref{eq:def vj-again} (and also in~\eqref{eq:inductive statement} and~\eqref{eq:M alpha}) does not vanish since we are assuming in Lemma~\ref{lem:induction} that $\{Ae_j\}_{j=1}^m$ are linearly independent. Define $\Omega\subset \{-1,1\}^{\s\setminus \tau}$ as follows.
\begin{equation}\label{eq:def Omega}
\Omega\eqdef \bigg\{\e\in \{-1,1\}^{\s\setminus \tau}:\ \bigg\|\sum_{i\in \s\setminus \tau} \e_i \mathsf{v}_i\bigg\|_2 \le M\sqrt{2|\s\setminus \tau|}\bigg\}.
\end{equation}
By the parallelogram identity we have
\begin{multline}\label{eq:use markov omega}
M^2|\s\setminus\tau|\stackrel{\eqref{eq:M alpha}}{\ge} \sum_{i\in \s\setminus \tau} \frac{1}{\|\proj_{F_i}Ae_i\|_2^2}\stackrel{\eqref{eq:def vj-again}}{=} \sum_{i\in \s\setminus \tau}\|\mathsf{v}_i\|_2^2=\frac{1}{2^{|\s\setminus\tau}|}\sum_{\e\in \{-1,1\}^{\s\setminus\tau}}\bigg\|\sum_{i\in \s\setminus \tau} \e_i \mathsf{v}_i\bigg\|_2^2\\
\stackrel{\eqref{eq:def Omega}}{>} \frac{1}{2^{|\s\setminus\tau|}}\sum_{\stackrel{\e\in \{-1,1\}^{\s\setminus\tau}}{\e\notin \Omega}} 2M^2|\s\setminus\tau|=2M^2|\s\setminus\tau|\left(1-\frac{|\Omega|}{2^{|\s\setminus\tau|}}\right).
\end{multline}
Since $|\s\setminus\tau|>0$, it follows from~\eqref{eq:use markov omega} that $|\Omega|> 2^{|\s\setminus \tau|-1}$.

We can now apply the Sauer--Shelah lemma, i.e., Lemma~\ref{sauer-shelah}, thus deducing that there exists a subset $\beta\subset \s\setminus \tau$ with $|\beta|\ge |\s\setminus \tau|/2$ such that $\proj_{\R^\beta}\Omega=\{-1,1\}^\beta$. Defining $\tau^*=\tau\cup\beta$ we shall now proceed to show that $\tau^*$ satisfies the inductive hypothesis with $k$ replaced by $k+1$.

Since $\beta\cap\tau=\emptyset$, $\tau\subset\s$ and $|\beta|\ge |\s\setminus \tau|/2$ we have
\begin{equation}\label{eq:tau star large}
|\tau^*|=|\tau|+|\beta|\ge |\tau|+\frac{|\s|-|\tau|}{2}=\frac{|\tau|+|\s|}{2}\ge \frac{(1-2^{-k})|\s|+|\s|}{2}=(1-2^{-k-1})|\s|.
\end{equation}
Next, suppose that $\vartheta\subset\m$ satisfies $\vartheta\supset \tau^*$. If $a\in \R^m$ then because $\proj_{\R^\beta}\Omega=\{-1,1\}^\beta$ there exists $\e\in \Omega$ such that for every $j\in \beta$ we have $\e_j=\sign(a_j)$. The fact that $\e\in \Omega$ means that
\begin{equation}\label{eq:use eps in Omega}
\bigg\|\sum_{i\in \s\setminus \tau} \e_i \mathsf{v}_i\bigg\|_2\le M\sqrt{2|\s\setminus \tau|}\le \frac{M\sqrt{2|\s|}}{2^{k/2}},
\end{equation}
where in the last step of~\eqref{eq:use eps in Omega} we used the fact that $|\tau|\ge (1-2^{-k})|\s|$.

The definition~\eqref{eq:def vj-again} of $\{\mathsf{v}_j\}_{j=1}^m$ implies that $\langle \mathsf{v}_i,Ae_j\rangle=\d_{ij}$ for every $i,j\in \m$. Hence,
\begin{multline}\label{eq:use sauer shelah}
\sum_{i\in \beta} |a_i|= \bigg\langle \sum_{i\in \beta} a_i Ae_i,\sum_{i\in \s\setminus \tau} \e_i \mathsf{v}_i \bigg\rangle =
\bigg\langle \sum_{i\in \vartheta} a_i Ae_i,\sum_{i\in \s\setminus \tau} \e_i \mathsf{v}_i \bigg\rangle-\sum_{i\in (\vartheta\setminus \beta)\cap (\s\setminus \tau)} \e_i a_i\\
\le \bigg\|\sum_{i\in \vartheta} a_i Ae_i\bigg\|_2 \bigg\|\sum_{i\in \s\setminus \tau} \e_i \mathsf{v}_i\bigg\|_2+\sum_{i\in \vartheta\cap (\s\setminus \tau^*)} |a_i|\stackrel{\eqref{eq:use eps in Omega}}{\le}
\frac{M\sqrt{2|\s|}}{2^{k/2}}\bigg\|\sum_{i\in \vartheta} a_i Ae_i\bigg\|_2+\sum_{i\in \vartheta\cap (\s\setminus \tau^*)} |a_i|.
\end{multline}
The penultimate step of~\eqref{eq:use sauer shelah} uses the Cauchy--Schwarz inequality and the fact that, by the definition of $\tau^*$, we have  $(\vartheta\setminus \beta)\cap (\s\setminus \tau)=\vartheta\cap (\s\setminus \tau^*)$. Now,
\begin{multline}\label{eq:use inductive hypothesis}
\sum_{i\in \tau^*} |a_i|=\sum_{i\in \tau} |a_i|+\sum_{i\in \beta} |a_i|\stackrel{\eqref{eq:inductive statement shorter}}{\le}
 \alpha_kM\sqrt{|\s|}\bigg\|\sum_{i\in \vartheta}a_i Ae_i\bigg\|_2+(2^k-1)\sum_{i\in \vartheta\cap(\s\setminus\tau)}|a_i|+\sum_{i\in \beta} |a_i|\\
= \alpha_kM\sqrt{|\s|}\bigg\|\sum_{i\in \vartheta}a_i Ae_i\bigg\|_2+(2^k-1)\sum_{i\in \vartheta\cap(\s\setminus\tau^*)}|a_i|+2^k\sum_{i\in \beta} |a_i|,
\end{multline}
where for the last step of~\eqref{eq:use inductive hypothesis} recall that $\vartheta\cap(\s\setminus\tau)=(\vartheta\cap(\s\setminus\tau^*))\cupdot \beta$. It remains to combine~\eqref{eq:use sauer shelah} and~\eqref{eq:use inductive hypothesis} to deduce that
\begin{equation}\label{eq:ss conclusion}
\sum_{i\in \tau^*} |a_i|\le \left(\alpha_k+2^{\frac{k+1}{2}}\right)M\sqrt{|\s|}\bigg\|\sum_{i\in \vartheta}a_i Ae_i\bigg\|_2+(2^{k+1}-1)\sum_{i\in \vartheta\cap(\s\setminus\tau^*)}|a_i|.
\end{equation}
Recalling the definition of $\alpha_k$ in~\eqref{eq:M alpha}, we have $\alpha_{k+1}=\alpha_k+2^{(k+1)/2}$, so the validity of~\eqref{eq:tau star large} and~\eqref{eq:ss conclusion} completes the proof that $\tau^*$ satisfies the inductive hypothesis with $k$ replaced by $k+1$.
\end{proof}

\begin{lemma}\label{lem:combine inductive with gro pie}  Fix $m,n,t\in \N$ and $\beta\subset \m$. Let $A:\R^m \to \R^n$ be a linear operator such that the vectors $\{Ae_j\}_{j=1}^m\subset \R^n$ are linearly independent. Then there exist two subsets  $\s,\tau\subset \beta$ satisfying $\s\subset \tau$,  $|\tau|\ge (1-2^{-t})|\beta|$ and $|\tau\setminus \s|\le |\beta|/4$ such that if we denote $\vartheta=\tau\cup(\m\setminus \beta)$ then
$$
\left\|\proj_{\R^\s}(AJ_\vartheta)^{-1}\right\|_{\S_\infty} \lesssim \max_{j\in \m} \frac{2^{\frac{t}{2}}}{\|\proj_{F_j}Ae_j\|_2},
$$
where we recall that the definition of the subspace $F_j\subset \R^n$ is given in~\eqref{eq:def Fj}.
\end{lemma}

\begin{proof} An application of Lemma~\ref{lem:induction}  with $\s=\beta$ and $k=t$ produces $\tau\subset \beta$ with $|\tau|\ge (1-2^{-t})|\beta|$ such that if we choose $\vartheta=\tau\cup(\m\setminus \beta)$ in~\eqref{eq:inductive statement} and continue with the notation in~\eqref{eq:M alpha} then
\begin{equation}\label{eq:use inductive lemma}
\forall\, a\in \R^m,\qquad \sum_{i\in \tau} |a_i|\lesssim 2^{\frac{t}{2}}M\sqrt{|\beta|} \bigg\|\sum_{i\in \vartheta} a_i Ae_i\bigg\|_2.
\end{equation}
Note that the above choice of $\vartheta$ makes the second term in the right hand side of~\eqref{eq:inductive statement}  vanish, and this is the only way by which~\eqref{eq:inductive statement} will be used here. However, the more complicated form of~\eqref{eq:inductive statement}  was needed in Lemma~\ref{lem:induction} to allow for the inductive construction to go through.

A different way to state~\eqref{eq:use inductive lemma} is the following operator norm bound.
\begin{equation*}\label{eq:deduce 2 to 1 norm}
\left\|\proj_{\R^\tau} (AJ_\vartheta)^{-1}\right\|_{\ell_2^\vartheta\to \ell_1^\tau}\lesssim 2^{\frac{t}{2}}M\sqrt{|\beta|}.
\end{equation*}
Since $|\tau|\ge (1-2^{-t})|\beta|\ge |\beta|/2$, if we set $\e\eqdef |\beta|/(4|\tau|)$ then $\e\in (0,1/2)$. We are therefore in position to use Lemma~\ref{lem:grothendieck pietsch}, thus producing a subset $\s\subset \tau$ with $|\tau\setminus \s|\le \e|\tau|=|\beta|/4$ such that
\begin{equation*}
\left\|\proj_{\R^\s}(AJ_\vartheta)^{-1}\right\|_{\S_\infty}=\left\|\proj_{\R^\s}\proj_{\R^\tau} (AJ_\vartheta)^{-1}\right\|_{\S_\infty}\lesssim \frac{2^{\frac{t}{2}}M\sqrt{|\beta|}}{\sqrt{\e|\tau|}}\asymp 2^{\frac{t}{2}}M. \qedhere
\end{equation*}
\end{proof}

\begin{proof}[Proof of Theorem~\ref{thm:full column rank}] Recall that, in the setting of Theorem~\ref{thm:full column rank}, we are currently assuming without loss of generality that $\omega=\m$.  Choose $r\in \N\cup \{0\}$ such that
\begin{equation}\label{eq:r chouce}
\frac{1}{2^{2r+1}}\le 1-\frac{k}{m}\le \frac{1}{2^{2r-1}}.
\end{equation}
 Denote $\tau_0\eqdef \m$ and $\s_0\eqdef \emptyset$. We shall construct by induction on $u\in \{0,\ldots, r+1\}$ two subsets $\sigma_{u},\tau_{u}\subset \m$ such that if we denote
 \begin{equation}\label{eq:def omega beta u}
 \beta_u\eqdef \tau_u\setminus \s_u\qquad \mathrm{and}\qquad \forall\, u\in \{1,\ldots,r+1\},\qquad \vartheta_u\eqdef \tau_u\cup\left(\m\setminus \beta_{u-1}\right),
 \end{equation}
  then the following properties hold true for every $u\in \{1,\ldots,r+1\}$.
\begin{enumerate}[label=(\alph*)]
\item\label{item:1} $\s_u\subset \tau_u\subset \beta_{u-1}$.
\item\label{item:2}  $|\tau_u|\ge (1-2^{-2r+u-4})|\beta_{u-1}|$ and $|\beta_{u}|\le \frac14 |\beta_{u-1}|$.
\item\label{item:3}  $\left\|\proj_{\R^{\s_u}}(AJ_{\vartheta_u})^{-1}\right\|_{\S_\infty} \lesssim 2^{r-\frac{u}{2}}M$, where $M$ is defined in~\eqref{eq:M alpha}.
\end{enumerate}
Indeed, assuming inductively that $\s_{u-1},\tau_{u-1}$ have been constructed, the existence of sets $\s_u,\tau_u$ with the desired properties follows from an application of Lemma~\ref{lem:combine inductive with gro pie} with $\beta=\beta_{u-1}$ and $t=2r-u+4$.


Recalling~\eqref{eq:def omega beta u}, by~\ref{item:1} we have $\beta_{u-1}=\beta_{u}\cupdot \sigma_u\cupdot (\beta_{u-1}\setminus \tau_u)$ for every $u\in \{1,\ldots,r+1\}$. Hence,
\begin{equation}\label{eq:betau to telescope}
|\s_u|=|\beta_{u-1}|-|\beta_u|-|\beta_{u-1}\setminus \tau_u|\ge |\beta_{u-1}|-|\beta_u|-\frac{|\beta_{u-1}|}{2^{2r-u+4}}\ge |\beta_{u-1}|-|\beta_u|-\frac{m}{2^{2r+u+2}},
\end{equation}
where the penultimate inequality in~\eqref{eq:betau to telescope} uses the first assertion in~\ref{item:2}  and the final inequality in~\eqref{eq:betau to telescope}  uses the fact that, by induction,  the second assertion in~\ref{item:2}  implies that $|\beta_{u-1}|\le m/4^{u-1}$, since $\beta_0=\m$. Observe that the sets $\{\s_u\}_{u=1}^{r+1}$ are pairwise disjoint, so if we denote
\begin{equation}\label{eq:sigma disjoint union}
\s\eqdef \bigcupdot_{u=1}^{r+1} \s_u,
\end{equation}
then
\begin{equation}\label{eq:sigma lower k}
|\s|=\sum_{u=1}^{r+1}|\s_u|\stackrel{\eqref{eq:betau to telescope}}{\ge} |\beta_0|-|\beta_{r+1}|-\frac{m}{2^{2r+2}}\sum_{u=1}^{\infty}\frac{1}{2^u}\ge m-\frac{m}{4^{r+1}}-\frac{m}{2^{2r+2}}=m-\frac{m}{2^{2r+1}}\stackrel{\eqref{eq:r chouce}}{\ge} k.
\end{equation}

Next, recalling the definition of $\vartheta_u$ in~\eqref{eq:def omega beta u}, observe that
\begin{equation}\label{eq:sigma in all omegas}
\s\subset \bigcap_{u=1}^{r+1} \vartheta_u.
\end{equation}
Indeed, in order to verify the validity of~\eqref{eq:sigma in all omegas} note that due to~\ref{item:1}  we have $\s_u,\s_{u+1},\ldots,\s_{r+1}\subset \tau_u$ and $\s_1,\ldots,\s_{u-1}\subset \m\setminus \beta_{u-1}$ for every $u\in \{1,\ldots,r+1\}$. It follows from~\eqref{eq:sigma in all omegas} that if $a\in \R^\s$ then for every $u\in \{1,\ldots,r+1\}$ we have $J_\s a\in J_{\vartheta_u}\R^{\vartheta_u}\subset \R^m$. Consequently,
\begin{equation}\label{projection commutation}
\proj_{\R^{\s_u}}(AJ_{\vartheta_u})^{-1}(A J_\s)a=\proj_{\R^{\s_u}}J_\s a.
\end{equation}
We therefore have the following estimate.
\begin{multline}\label{eq:for all a on r sigma}
\|J_\s a\|_2^2\stackrel{\eqref{eq:sigma disjoint union}}{=} \bigg\|\sum_{u=1}^{r+1}\proj_{\R^{\s_u}} J_\s a\bigg\|_2^2=\sum_{u=1}^{r+1} \left\|\proj_{\R^{\s_u}} J_\s a\right\|_2^2
\stackrel{\eqref{projection commutation}}{=}\sum_{u=1}^{r+1} \left\|\proj_{\R^{\s_u}}(AJ_{\vartheta_u})^{-1}(A J_\s)a\right\|_2^2\\\stackrel{\ref{item:3}}{\lesssim}\sum_{u=1}^{r+1} 2^{2r-u}M^2\left\|(A J_\s)a\right\|_2^2\asymp 2^{2r}M^2\left\|(A J_\s)a\right\|_2^2 \stackrel{\eqref{eq:r chouce}}{\asymp} \frac{mM^2}{m-k} \left\|(A J_\s)a\right\|_2^2.
\end{multline}
Recalling the definition of $M$ in~\eqref{eq:M alpha}, since~\eqref{eq:for all a on r sigma} holds true for every $a\in \R^\s$ we conclude that
$$
\left\|(AJ_\s)^{-1}\right\|_{\S_\infty}\lesssim \frac{\sqrt{m}}{\sqrt{m-k}}\cdot \max_{j\in \m} \frac{1}{\|\proj_{F_j}Ae_j\|_2}.
$$
This is the desired estimate~\eqref{eq:max projection}, which, together with~\eqref{eq:sigma lower k}, concludes the proof of Theorem~\ref{thm:full column rank}.
\end{proof}

\subsection{Geometric interpretation of Theorem~\ref{thm:full column rank}}\label{sec:geometric} Theorem~\ref{thm:gia ellipsoid} below is a result of Giannopoulos~\cite{Gia96}. It can be viewed as a geometric analogue of the Sauer--Shelah lemma for ellipsoids. The (rough) analogy between the two results is that they both assert that certain ``large" subsets of $\R^n$ must admit a large rank coordinate projection that contains a certain ``canonical shape" (a full hypercube in the Sauer--Shelah case and a large  Euclidean ball in Giannopoulos' case). A different geometric analogue of the Sauer--Shelah lemma was proved by Szarek and Talagrand in~\cite{ST89}.

\begin{theorem}[Giannopoulos]\label{thm:gia ellipsoid} There exists a universal constant $c\in (0,\infty)$ with the following property. Suppose that $m,n\in \N$ and $\e\in (0,1)$. Let $y_1,\ldots,y_m\in \R^n$ be vectors that satisfy $\|y_i\|_2\le 1$ for every $i\in \m$. Denote
\begin{equation}\label{eq:def ellipsoid}
\mathcal{E}\eqdef \bigg\{a=(a_1,\ldots,a_m)\in \R^m;\ \bigg\|\sum_{j=1}^m a_jy_j\bigg\|_2\le 1\bigg\}.
\end{equation}
Then there exists a subset $\s\subset \m$ with $|\s|\ge (1-\e)m$ such that $\proj_{\R^\s}(\mathcal{E})\supseteq c\sqrt{\e} B_2^\s$, where $B_2^\s=\{x\in \R^\s:\ \|x\|_2\le 1\}$ denotes the unit Euclidean ball in $\R^\s$.
\end{theorem}
In this section we shall show that Theorem~\ref{thm:gia ellipsoid} is equivalent to Theorem~\ref{thm:full column rank}, thus in particular describing a shorter proof of  Theorem~\ref{thm:full column rank} that relies on Theorem~\ref{thm:gia ellipsoid}.

Let us first prove that Theorem~\ref{thm:full column rank} implies Theorem~\ref{thm:gia ellipsoid}. Suppose that we are in the setting that is described in the statement of Theorem~\ref{thm:gia ellipsoid}. It was observed in~\cite{Gia96} that Theorem~\ref{thm:gia ellipsoid} with the additional assumption that $y_1,\ldots,y_m$ are linearly independent formally implies Theorem~\ref{thm:gia ellipsoid} in the above stated generality. Indeed, this follows by applying (the linear independent case of) Theorem~\ref{thm:gia ellipsoid} to the linearly independent vectors $y_1+e_{n+1},y_2+e_{n+2},\ldots,y_m+e_{n+m}\in \R^{n+m}$. So, suppose that $y_1,\ldots,y_m\in \R^n$ are linearly independent and let $x_1,\ldots,x_m\in \spn\{y_1,\ldots,y_m\}$ be the corresponding dual basis, i.e.,
\begin{equation}\label{eq:dual for quoting}
\forall\, i,j\in \m,\qquad \langle x_i,y_j\rangle =\delta_{ij}.
\end{equation}
Define a linear operator $A:\R^m\to \R^n$ by setting $Ae_i=x_i$ for every $i\in \m$. Continuing with the notation for the subspace $F_j\subset \R^n$ that is given in~\eqref{eq:def Fj} (with $\omega=\m$), we know by~\eqref{eq:dual for quoting} that $y_j\in F_j$, so $\langle \proj_{F_j}x_j,y_j\rangle=\langle x_j,y_j\rangle=1$. Since we are assuming in the setting of Theorem~\ref{thm:gia ellipsoid} that $\|y_j\|_2\le 1$, this implies that $1=\langle \proj_{F_j}x_j,y_j\rangle\le \|y_j\|_2\cdot \|\proj_{F_j}x_j\|_2\le \|\proj_{F_j}x_j\|_2$.

An application of Theorem~\ref{thm:full column rank}  now shows that there exists $\s\subset \m$ with $|\s|\ge \lfloor (1-\e)m\rfloor$ and a universal constant $c\in (0,\infty)$ such that
\begin{equation}\label{eq:use theorem b}
\forall\, b\in \R^\s,\qquad \bigg\|\sum_{j\in \s} b_jx_j\bigg\|_2\ge c\sqrt{\e} \bigg(\sum_{j\in \s} b_j^2\bigg)^{\frac12}.
\end{equation}
We claim that~\eqref{eq:use theorem b} implies that $\proj_{\R^\s}(\mathcal{E})\supseteq c\sqrt{\e} B_2^\s$, where $\mathcal{E}$ is given in~\eqref{eq:def ellipsoid}. Indeed, suppose that $a=\sum_{j\in \s}a_je_j\in \R^\s$ satisfies
\begin{equation}\label{eq:a small norm}
a\in c\sqrt{\e}B_2^\s\iff \bigg(\sum_{j\in \s} a_j^2\bigg)^{\frac12}\le c\sqrt{\e}.
\end{equation}
Since the vectors $\{x_j\}_{j\in \s}\cup\{y_j\}_{j\in \m\setminus\s}$ form a basis of $\spn\{y_1,\ldots,y_m\}$, there exists a vector $b=(b_1,\ldots,b_m)\in \R^m$ such that
\begin{equation}\label{eq:mixed with dual}
\sum_{j\in \s} a_jy_j=\sum_{j\in \s} b_j x_j+\sum_{j\in \m\setminus\s} b_jy_j.
\end{equation}
Denote
\begin{equation}\label{eq:def a star}
a^*=(a_1^*,\ldots,a_m^*)\eqdef \sum_{j\in \s} a_je_j-\sum_{j\in \m\setminus\s} b_je_j\in \R^m.
\end{equation}
Then $\proj_{\R^\s} a^*=a$ and
\begin{multline}\label{eq:to cancel norm}
\bigg\|\sum_{j=1}^m a_j^*y_j\bigg\|_2^2=\bigg\langle\sum_{j=1}^m a_j^*y_j,\sum_{j=1}^m a_j^*y_j\bigg\rangle \stackrel{\eqref{eq:mixed with dual}\wedge \eqref{eq:def a star}}{=}\bigg\langle\sum_{j=1}^m a_j^*y_j,\sum_{j\in \s} b_j x_j\bigg\rangle \stackrel{\eqref{eq:dual for quoting}\wedge \eqref{eq:def a star}}{=}\sum_{j\in \s} a_jb_j\\
\le \bigg(\sum_{j\in \s} a_j^2\bigg)^{\frac12}\bigg(\sum_{j\in \s} b_j^2\bigg)^{\frac12}\stackrel{\eqref{eq:a small norm}}{\le} c\sqrt{\e} \bigg(\sum_{j\in \s} b_j^2\bigg)^{\frac12}\stackrel{\eqref{eq:use theorem b}}{\le} \bigg\|\sum_{j\in \s} b_jx_j\bigg\|_2\stackrel{\eqref{eq:mixed with dual}\wedge \eqref{eq:def a star} }{=}\bigg\|\sum_{j=1}^m a_j^*y_j\bigg\|_2.
\end{multline}
By cancelling $\big\|\sum_{j=1}^m a_j^*y_j\big\|_2$ from both sides of~\eqref{eq:to cancel norm} and recalling~\eqref{eq:def ellipsoid}, we conclude that $a^*\in \mathcal{E}$.  Thus $a=\proj_{\R^\s} a^*\in \proj_{\R^\s}(\mathcal{E})$, as required.

Next, we shall prove the converse implication, i.e., that Theorem~\ref{thm:gia ellipsoid} implies Theorem~\ref{thm:full column rank}. Suppose that we are in the setting of Theorem~\ref{thm:full column rank}. As we explained in the beginning of Section~\ref{sec:gia}, we may assume without loss of generality that $\omega=\m$, hence $\rank(A)=m$. Let $M\in (0,\infty)$ be defined as in~\eqref{eq:M alpha}, i.e., $M=\max_{j\in \m} \|\proj_{F_j}Ae_j\|_2^{-1}$. Set
$$
\forall\, i\in \m,\qquad {y}_i\eqdef \frac{\proj_{F_i}Ae_i}{\|\proj_{F_i}Ae_i\|_2}\in \R^n.
$$
Then by definition $\|y_i\|_2=1$ for every $j\in \m$, and, by the same reasoning as in the beginning of Section~\ref{sec:dual basis}, we know that $\langle y_j,Ae_j\rangle\ge 1/M$ and $\langle y_i,Ae_j\rangle =0$ for every distinct $i,j\in \m$. By Theorem~\ref{thm:gia ellipsoid} applied with $\e=1-k/m$ there exists $\s\subset \m$ of size $|\s|\ge (1-\e)m=k$ such that $\proj_{\R^\s} (\mathcal{E})\supseteq c\sqrt{\e}B_2^\s$, where $\mathcal{E}$ is defined in~\eqref{eq:def ellipsoid}. Suppose that $a\in \R^\s\setminus\{0\}$. Then $c\sqrt{\e} a/\|a\|_2\in \proj_{\R^\s} (\mathcal{E})$, which means that there exists $b\in \R^m$ such that $b_j=c\sqrt{\e}a_j/\|a\|_2$ for every $j\in \s$ and (by the definition of $\mathcal{E}$) we have $\big\|\sum_{i=1}^m b_i y_i\|_2\le 1$.  So,
\begin{multline*}
\bigg\|\sum_{j\in \s} a_jAe_j\bigg\|_2\ge \bigg\|\sum_{j\in \s} a_jAe_j\bigg\|_2\cdot \bigg\|\sum_{j=1}^m b_jy_j \bigg\|_2\ge \bigg\langle\sum_{j\in \s} a_jAe_j,\sum_{j=1}^m b_jy_j \bigg\rangle\\=\sum_{j\in\s} a_jb_j\langle Ae_j,y_j\rangle= \sum_{j\in\s} \frac{c\sqrt{\e} a_j^2}{\|a\|_2}\langle Ae_j,y_j\rangle\ge \frac{c\sqrt{\e}}{M\|a\|_2}\sum_{j\in \s} a_j^2=\frac{c\sqrt{m-k}}{M\sqrt{m}}\|a\|_2.
\end{multline*}
This is precisely the desired conclusion in Theorem~\ref{thm:full column rank}.\qed

\section{Marcus--Spielman--Srivastava}\label{sec:MSS}

Our goal here is to prove Theorem~\ref{thm:MSS version}. This section differs from the previous sections in that we shall use the method of interlacing polynomials of Marcus--Spielman--Srivastava without sketching the proofs of the tools that we quote. The reason for this is that the ideas of Marcus--Spielman--Srivastava are remarkable and deep, but nevertheless elementary and accessible, and their presentation in~\cite{MSS15-ramanujan,MSS15-kadison} and especially in the beautiful survey~\cite{MSS14} (which is the main reference in the present section) is already a perfect exposition for a wide mathematical audience.

Suppose that $A:\R^m \to \R^n$ is a linear operator. Let $\j_1,\ldots,\j_k$ be i.i.d. random variables that are distributed uniformly over $\m$. For every $t\in \{1,\ldots,k\}$ consider the random vector
\begin{equation}\label{eq:def wt}
\w_t\eqdef \sqrt{m}Ae_{\j_t}.
\end{equation}
Then,
\begin{equation}\label{eq:wt tensor expectation}
\E\big[\w_t\otimes \w_t\big]=\sum_{i=1}^m (Ae_i)\otimes (Ae_i)=AA^*.
\end{equation}

Denote
\begin{equation}\label{eq:def rho}
\gamma\eqdef \frac{\rank(A)\left(\sqrt{\rank(A)}-\sqrt{k}\right)^2}{\sum_{i=1}^{\rank(A)}\frac{1}{\ss_i(A)^2}}.
\end{equation}
With this notation, we shall prove below that
\begin{equation}\label{eq:positive probability}
\Pr\bigg[\ss_k\bigg(\sum_{t=1}^k \w_t\otimes \w_t\bigg)\ge \gamma\bigg]>0.
\end{equation}
Recalling~\eqref{eq:def wt}, we see that~\eqref{eq:def rho} and~\eqref{eq:positive probability} imply that there exist $j_1,\ldots,j_k\in \m$ such that
\begin{equation}\label{eq:apply probability}
\ss_k\bigg(\sum_{t=1}^k (Ae_{j_t})\otimes (Ae_{j_t})\bigg)\ge \frac{\gamma}{m}=\frac{\rank(A)\left(\sqrt{\rank(A)}-\sqrt{k}\right)^2}{m\sum_{i=1}^{\rank(A)}\frac{1}{\ss_i(A)^2}}.
\end{equation}
The rank of the operator $B\eqdef\sum_{t=1}^k (Ae_{j_t})\otimes (Ae_{j_t})$ is at most the cardinality of $\s\eqdef \{j_1,\ldots,j_k\}$. At the same time, by~\eqref{eq:apply probability} we know that $\ss_k(B)>0$, because we are assuming that $k<\rank(A)$. Thus $B$ has rank at least $k$, implying that the indices $j_1,\ldots,j_k$ are necessarily distinct, or equivalently that $|\s|=k$. Consequently $B=(AJ_\s)(AJ_\s)^*$ and $\ss_k(B)=\ss_{\min}(B)=\ss_{\min}(AJ_\s)^2=1/\|(AJ_\s)^{-1}\|_{\S_\infty}^2$. Therefore~\eqref{eq:apply probability} is the same as the desired restricted invertibility statement~\eqref{eq:with average projection} of Theorem~\ref{thm:MSS version}.

It remains to establish the validity of~\eqref{eq:positive probability}. Denote $Q\eqdef AA^*:\R^n\to \R^n$ and let $\q:\R\to \R$ be the polynomial that is defined as follows.
$$
\forall\, x\in \R,\qquad \q(x)\eqdef \left.(I-\partial_y)^k\det(xI_n+yQ)\right|_{y=0},
$$
where $I$ denotes the identity operator on the space of polynomials and $\partial_y$ is the differentiation operator with respect to the variable $y$ (and, as before, $I_n$ is the $n$ by $n$ identity matrix). By Theorem~4.1 in~\cite{MSS14}, the degree $n$ polynomial  $\q$ is the expectation of the characteristic polynomial of the random matrix $\sum_{t=1}^k \w_t\otimes \w_t$. By Theorem~4.5 in~\cite{MSS14}, all the roots of $\q$ are real, and we denote their decreasing rearrangement by $\rho_1\ge \rho_2\ge\ldots\ge \rho_n$. Thus, $\rho_k$ is the $k$'th largest root of $\q$. A combination of Theorem~1.7 in~\cite{MSS14} and Theorem~4.1 in~\cite{MSS14} shows that
\begin{equation}\label{eq:prob with rho k}
\Pr\bigg[\ss_k\bigg(\sum_{t=1}^k \w_t\otimes \w_t\bigg)\ge \rho_k\bigg]>0.
\end{equation}
Consequently, in order to prove~\eqref{eq:positive probability} it suffices to prove that $\rho_k\ge \gamma$, where $\gamma$ is defined in~\eqref{eq:def rho}.

Write $Q=U\Delta U^{-1}$, where $U:\R^n\to \R^n$ is an orthogonal matrix and $\Delta:\R^n\to \R^n$ is a diagonal matrix whose diagonal equals $(\ss_1(A)^2,\ldots,\ss_n(A)^2)\in \R^n$. Then for every $x,y\in \R$ we have
$$
\det(xI_n+yQ)=\det\left(U(xI_n+y\Delta)U^{-1}\right)=\prod_{i=1}^n \left(x+y\ss_i(A)^2\right)=x^{n-\rank(A)}\prod_{i=1}^{\rank(A)} \left(x+y\ss_i(A)^2\right),
$$
where we used the fact that $\ss_i(A)=0$ when $i>\rank(A)$.  Consequently,
\begin{equation}\label{eq:bring det in}
\q(x)=x^{n-\rank(A)}(I-\partial_y)^k \prod_{i=1}^{\rank(A)} \left(x+y\ss_i(A)^2\right)\Big|_{y=0}.
\end{equation}
We claim that if we denote by $\D$ the differentiation operator on the space of polynomials then
\begin{equation}\label{eq:with D operator}
\q(x)= x^{n-k} \prod_{i=1}^{\rank(A)} \left(I-\ss_i(A)^2\D\right)x^k.
\end{equation}
The identity~\eqref{eq:with D operator} is proven in the special case $\ss_1(A)=\ldots=\ss_{\rank(A)}(A)=1$ in~\cite{MSS14}.  The validity of~\eqref{eq:with D operator} in full generality follows from checking that the coefficients of the polynomials that appear in the right hand sides of~\eqref{eq:bring det in} and~\eqref{eq:with D operator} are equal to each other. Indeed, starting with~\eqref{eq:bring det in},
\begin{align}\label{eq:first expand}
\nonumber x^{n-\rank(A)}(I-\partial_y)^k &\prod_{i=1}^{\rank(A)} \left(x+y\ss_i(A)^2\right)\Big|_{y=0}\\\nonumber  &=x^{n-\rank(A)}\sum_{u=0}^k\binom{k}{u}(-1)^u\partial_y^u \sum_{\Omega\subset \{1,\ldots,\rank(A)\}} x^{\rank(A)-|\Omega|}y^{|\Omega|}\prod_{i\in \Omega} \ss_i(A)^2\Big|_{y=0}\\
&=\sum_{\substack{\Omega\subset \{1,\ldots,\rank(A)\}\\|\Omega|\le k}}\frac{(-1)^{|\Omega|}x^{n-|\Omega|}k!}{(k-|\Omega|)!} \prod_{i\in \Omega} \ss_i(A)^2,
\end{align}
since $\partial_y^u y^{|\Omega|}|_{y=0}=|\Omega|!\cdot \1_{\{|\Omega|=u\}}$ for every $(u,\Omega)\in \{0,\ldots,k\}\times \{1,\ldots,\rank(A)\}$. At the same time,
\begin{equation}\label{eq:second expand}
x^{n-k} \prod_{i=1}^{\rank(A)} \left(I-\ss_i(A)^2\D\right)x^k=x^{n-k}\sum_{\Omega\subset\{1,\ldots,\rank(A)\}} (-1)^{|\Omega|} \bigg(\prod_{i\in \Omega} \ss_i(A)^2 \bigg)\D^{|\Omega|} x^k.
\end{equation}
Since for every for every $(u,\Omega)\in \{0,\ldots,k\}\times \{1,\ldots,\rank(A)\}$ we have $\D^{|\Omega|} x^k=0$ if $|\Omega|>k$ and $\D^{|\Omega|} x^k=x^{k-|\Omega|}k!/(k-|\Omega|)!$ if $|\Omega|\le k$, the validity of~\eqref{eq:with D operator} follows by comparing~\eqref{eq:first expand} and~\eqref{eq:second expand}.

Having established the identity~\eqref{eq:with D operator}, we shall proceed to prove the desired estimate $\rho_k\ge \gamma$ by applying the barrier method of~\cite{BSS12}, reasoning along the lines of the argument that is presented in~\cite{MSS14}. Following~\cite{BSS12,SV13}, given a polynomial $f:\R\to \R$ and $\phi\in (0,\infty)$ we consider the corresponding ``soft spectral edge" $\mathrm{\bf smin}_\phi(f)\in\R$, which is defined as follows
\begin{equation}\label{eq:def smin}
\mathrm{\bf smin}_{\phi}(f)\eqdef \inf\left\{b\in \R:\ f'(b)=-\phi f(b)\right\}.
\end{equation}
As explained in~\cite[Section~3.2]{MSS14}, it is simple to check that for every $\phi\in (0,\infty)$ the smallest real root of $f$ is at least the quantity $\mathrm{\bf smin}_{\phi}(f)$. Hence, if we define \begin{equation}\label{eq:def g prod}
g(x)\eqdef \prod_{i=1}^{\rank(A)} \left(I-\ss_i(A)^2\D\right)x^k,
\end{equation}
then it follows from the above discussion and the identity~\eqref{eq:with D operator} that it suffices to prove that
\begin{equation}\label{eq:goal g}
\sup_{\phi\in (0,\infty)}  \mathrm{\bf smin}_{\phi}(g)\ge \gamma.
\end{equation}
Indeed, by~\eqref{eq:with D operator} the $n$ real roots of $\q$ consist of $0$ with multiplicity $n-k$ and also the $k$ roots of $g$ (which are therefore necessarily real).  Since $g$ has degree $k$, the validity of~\eqref{eq:goal g} would imply that the smallest root of $g$ is at least $\gamma>0$, so the $k$'th largest root of $\q$ would be at least $\gamma$ as well.

To prove~\eqref{eq:goal g}, recall that Lemma~3.8 of~\cite{MSS14} asserts that for every  polynomial $f:\R\to \R$ all of whose roots are real, and for every $\phi\in (0,\infty)$, we have
\begin{equation}\label{eq:smin growth}
\mathrm{\bf smin}_\phi\big((I-\D)f\big)\ge \mathrm{\bf smin}_{\phi}(f)+\frac{1}{1+\phi}.
\end{equation}
For $\ss\in (0,\infty)$ define $f_\ss:\R\to \R$ by setting $f_{\ss}(x)\eqdef f(\ss x)$ for every $x\in \R$. Observe that
\begin{equation}\label{eq:rescaling identities}
\forall\, \ss\in (0,\infty),\qquad (I-\ss \D)f= ((I-\D)f_{\ss})_{1/\ss}\qquad \mathrm{and}\qquad \mathrm{\bf smin}_{\phi}(f_\ss)\stackrel{\eqref{eq:def smin}}{=} \frac{\mathrm{\bf smin}_{\phi/\ss}(f)}{\ss}.
\end{equation}
Consequently, for every real-rooted polynomial $f$ and every $\ss,\phi\in (0,\infty)$ we have
\begin{multline}\label{eq:smin growth rescaled}
\mathrm{\bf smin}_\phi\big((I-\ss \D)f\big)\stackrel{\eqref{eq:rescaling identities}}{=}\mathrm{\bf smin}_{\phi}\big(((I-\D)f_{\ss})_{1/\ss}\big)\stackrel{\eqref{eq:rescaling identities}}{=}\ss\cdot \mathrm{\bf smin}_{\ss \phi}\big((I-\D)f_\ss\big)\\\stackrel{\eqref{eq:smin growth}}{\ge} \ss\left(\mathrm{\bf smin}_{\ss \phi}(f_\ss)+\frac{1}{1+\ss\phi}\right)\stackrel{\eqref{eq:rescaling identities}}{=}\mathrm{\bf smin}_{\phi}(f)+\frac{1}{\frac{1}{\ss}+\phi}.
\end{multline}
By iterating~\eqref{eq:smin growth rescaled} we see that
\begin{multline}\label{eq:convexity of 1/x}
\mathrm{\bf smin}_{\phi}(g)\ge \mathrm{\bf smin}_{\phi}(x^k)+\sum_{i=1}^{\rank(A)} \frac{1}{\frac{1}{\ss_i(A)^2}+\phi}\\\stackrel{\eqref{eq:def smin}}{=} -\frac{k}{\phi}+\sum_{i=1}^{\rank(A)} \frac{1}{\frac{1}{\ss_i(A)^2}+\phi}\ge -\frac{k}{\phi}+\frac{\rank(A)}{\phi+\frac{1}{\rank(A)}\sum_{i=1}^{\rank(A)}\frac{1}{\ss_i(A)^2} },
\end{multline}
where the last step of~\eqref{eq:convexity of 1/x} holds true due to  the convexity of the function $x\mapsto 1/(\phi+x)$ on $(0,\infty)$. One can check that the value of $\phi$ that maximizes the right hand side of~\eqref{eq:convexity of 1/x} is
$$
\phi_{\max}\eqdef \frac{\sqrt{k}}{\sqrt{\rank(A)}-\sqrt{k}}\bigg(\frac{1}{\rank(A)}\sum_{i=1}^{\rank(A)}\frac{1}{\ss_i(A)^2}\bigg).
$$
The right hand side of~\eqref{eq:convexity of 1/x} equals $\gamma$ when $\phi=\phi_{\max}$, so $\rho_k\ge \mathrm{\bf smin}_{\phi_{\max}}(g)\ge \gamma$, as required.\qed

\begin{remark}
The above argument actually yields a subset $\s\subset\m$ with $|\s|=k$ such that
\begin{equation}\label{eq:with transform}
\ss_{\min}(AJ_\s)^2=\ss_{k}(AJ_\s)^2\ge \frac{1}{m}\sup\bigg\{-\frac{k}{\phi}+\sum_{i=1}^{\rank(A)} \frac{\ss_i(A)^2}{1+\phi\ss_i(A)^2}:\ \phi\in (0,\infty)\bigg\}.
\end{equation}
Indeed, continuing with the above notation, we explained why $\rho_k\ge \sup_{\phi\in (0,\infty)} \mathrm{\bf smin}_\phi (g)$, so~\eqref{eq:with transform} follows from~\eqref{eq:prob with rho k} and the penultimate step in~\eqref{eq:convexity of 1/x}.

The estimate~\eqref{eq:with transform} is more complicated than the assertion of Theorem~\ref{thm:MSS version}, but it is sometimes significantly stronger. One such instance is the matrix $A$ of Example~\ref{example:harmonic}. In that case, a somewhat tedious but straightforward  computation allows one to obtain sharp estimates on the right hand side of~\eqref{eq:with transform}, yielding bounds that coincide (up to constant factors) with those that are stated in Example~\ref{example:harmonic} as a consequence of Theorem~\ref{thm:main rank theorem}, while Theorem~\ref{thm:MSS version} yields much weaker bounds. There are also situations in which~\eqref{eq:with transform} yields worse bounds than those that follow from Theorem~\ref{thm:full column rank}, e.g.~when $\ss_1(A)\asymp\ldots\asymp \ss_m(A)\asymp 1$ and $k=(1-\e)m$ the bound on $\|(AJ_\s)^{-1}\|_{\S_\infty}$ that follows from~\eqref{eq:with transform} is $O(1/\e)$ while  in the same situation Theorem~\ref{thm:full column rank} yields the bound $\|(AJ_\s)^{-1}\|_{\S_\infty}\lesssim1/\sqrt{\e}$.
\end{remark}

\bibliographystyle{alphaabbrvprelim}
\bibliography{invertibility}

 \end{document}